\documentclass[11pt,reqno]{amsart}
\usepackage{hyperref}
\usepackage{amscd,amssymb,palatino}
\usepackage[utf8]{inputenc}
\usepackage[english]{babel}
\usepackage{caption}
\usepackage{etex}
\usepackage{amsmath}
\usepackage{amssymb}
\usepackage{color}
\usepackage{shadow}
\usepackage{epsfig}
\usepackage{epic}
\usepackage{graphics}
\usepackage{graphicx}
\usepackage{psfrag}

\usepackage{calc}
\usepackage{tikz}
\usepackage{tikz-3dplot}
\usepackage{tikz-cd}
\usepackage[dvipsnames,prologue,table]{pstricks}

\usepackage[cmtip,arrow]{xy}  
\usepackage{pb-diagram,pb-xy}  
\usepackage{graphicx}    
\usepackage{amsfonts}
\usepackage{amsmath, amssymb}
\usepackage{color}

\usepackage{graphicx}
\usepackage{url}
\usepackage{tikz}
\usetikzlibrary{patterns}
\usetikzlibrary{matrix, calc, decorations.pathmorphing, decorations.markings, shapes.arrows, shapes, snakes}
\usepackage{color}
\usepackage{pgfplots}
\usepackage{subfigure}
\usepackage{caption}
\usepackage{float}

\usepackage{amsmath}
\usepackage{amsthm}
\usepackage{amsfonts}
\usepackage{array}
\newcolumntype{C}[1]{>{\centering\arraybackslash}p{#1}}
\usetikzlibrary{arrows,decorations,patterns,positioning,automata,shadows,fit,shapes,calc,decorations.markings,backgrounds,scopes,decorations.text}

\usepackage[cmtip,arrow]{xy}  
\usepackage{pb-diagram,pb-xy}  

\newtheorem{theorem}{Theorem}
\newtheorem*{theorem*}{Theorem}
\newtheorem{prop}[theorem]{Proposition}
\newtheorem{corollary}[theorem]{Corollary}
\newtheorem*{prop*}{Proposition}
\newtheorem{lemma}[theorem] {Lemma}
\newtheorem{defi}{Definition}

\newtheorem{claim}{Claim}

\theoremstyle{definition}

\newtheorem{rem}{Remark}

\theoremstyle{plain}
\newtheorem*{ex}{Example}

\def\R{\mathbb{R}}

\newcommand{\pp}[2]{\frac{\partial#1}{\partial#2}}

\newcommand{\T}{\mathbb{T}}

\newcommand{\suchthat}{\;\ifnum\currentgrouptype=16 \middle\fi|\;}

\newcommand\restr[2]{{
  \left.\kern-\nulldelimiterspace
  #1
  \vphantom{\big|}
  \right|_{#2}
  }}

\newcommand\DrawGenus[7]{
  \pgfmathsetmacro{\xstart}{#1 - (0.985*#4)}
  \pgfmathsetmacro{\ystart}{#2 + (0.2*#3)}

	\draw[color = #6, rotate around={#5:(#1,#2)}, #7] (\xstart, \ystart) arc (190:350:#4  and #3);
	\draw[color = #6, rotate around={#5:(#1,#2)}, #7] (\xstart, \ystart) arc (190:210:#4  and #3) arc (150:30:#4  and #3) arc (330:350:#4  and #3);
}

\newcommand\DrawFilledGenus[8]{
  \pgfmathsetmacro{\xstart}{#1 - (0.985*#4)}
  \pgfmathsetmacro{\ystart}{#2 + (0.2*#3)}
	\draw[color = #6, rotate around={#5:(#1,#2)}, #7] (\xstart, \ystart) arc (190:350:#4  and #3);
	\draw[color = #6, rotate around={#5:(#1,#2)}, #7] (\xstart, \ystart) arc (190:210:#4  and #3) arc (150:30:#4  and #3) arc (330:350:#4  and #3);
	\draw[color = #6, rotate around={#5:(#1,#2)}, #7, fill = #8] (\xstart, \ystart) arc (190:210:#4  and #3) arc (150:30:#4  and #3) arc (-30:-150:#4  and #3);
}

\newcommand\DrawDonut[7]{
  \pgfmathsetmacro{\fctr}{.08}
  \pgfmathsetmacro{\newwidth}{0.5*#4}
  \pgfmathsetmacro{\newheight}{0.5*#3}
  \draw[color = #6, rotate around={#5:(#1,#2)}, #7] (#1, #2) ellipse (#4  and #3);
  \DrawGenus{#1}{#2}{\newheight}{\newwidth}{#5}{#6}{#7}
}

\newcommand\DrawFilledDonutops[8]{
  \pgfmathsetmacro{\fctr}{.08}
  \pgfmathsetmacro{\newwidth}{0.5*#4}
  \pgfmathsetmacro{\newheight}{0.5*#3}
  \draw[color = #6, rotate around={#5:(#1,#2)}, #7, fill = #6, opacity = .6] (#1, #2) ellipse (#4  and #3);
  \DrawFilledGenus{#1}{#2}{\newheight}{\newwidth}{#5}{#6}{#7}{#8}
}

\tikzstyle{mytheorembox} = [draw=vdgreen, fill=blue!20, very thick, rectangle, rounded corners, inner sep=10pt, inner ysep=10pt]
\tikzstyle{mytheoremfancytitle} =[fill=vdgreen, text=white]

\definecolor{vdblue}{rgb}{0,0,.3}
\definecolor{dblue}{rgb}{0,0,.7}
\definecolor{lblue}{rgb}{.3,.3,1}
\definecolor{vlblue}{rgb}{.7,.7,1}
\definecolor{vvlblue}{rgb}{.9,.9,1}

\definecolor{vdred}{rgb}{.3,0,0}
\definecolor{dred}{rgb}{.7,0,0}
\definecolor{lred}{rgb}{1,.3,.3}
\definecolor{vlred}{rgb}{1,.7,.7}

\definecolor{vdgreen}{rgb}{0,.2,0}
\definecolor{dgreen}{rgb}{0,.4,0}
\definecolor{lgreen}{rgb}{.3,1,.3}
\definecolor{vlgreen}{rgb}{.7,1,.7}

\definecolor{lyellow}{rgb}{1,1,.3}

\definecolor{gray1}{rgb}{0.22,0.22,0.22}
\definecolor{gray2}{rgb}{0.28,0.28,0.28}
\definecolor{gray3}{rgb}{0.36,0.36,0.36}
\definecolor{gray4}{rgb}{0.44,0.44,0.44}
\definecolor{gray5}{rgb}{0.52,0.52,0.52}
\definecolor{gray6}{rgb}{0.6,0.6,0.6}
\definecolor{gray7}{rgb}{0.68,0.68,0.68}
\definecolor{gray8}{rgb}{0.76,0.76,0.76}

\definecolor{color1}{rgb}{1,0,0}
\definecolor{color2}{rgb}{0.98,0,0.816}
\definecolor{color3}{rgb}{0.717,0,1}
\definecolor{color4}{rgb}{0,0,1}
\definecolor{color5}{rgb}{0,1,1}
\definecolor{color6}{rgb}{0,1,0}
\definecolor{color8}{rgb}{1,1,0}
\definecolor{color7}{rgb}{1,0.651,0}

\title[Integrable systems on singular symplectic manifolds:
From local to global]{Integrable systems on singular symplectic manifolds:\\
From local to global}

\author{Robert Cardona}
\address{{Laboratory of Geometry and Dynamical Systems, Department of Mathematics}, Universitat Polit\`{e}cnica de Catalunya and BGSMath, Barcelona, Spain }
\email{robert.cardona@upc.edu}
\author{Eva Miranda}
\thanks{{ Robert Cardona is supported by an FPI grant under the Maria de Maeztu-BGSMath excellence programme. Eva Miranda  is supported by the Catalan Institution for Research and Advanced Studies via an ICREA Academia Prize 2016. Both authors are supported by the grants reference number MTM2015-69135-P (MINECO/FEDER) and reference number 2017SGR932 (AGAUR) and the project PID2019-103849GB-I00 / AEI / 10.13039/501100011033. }}

\address{{Laboratory of Geometry and Dynamical Systems Department of Mathematics $\&$ Institut de Matemàtiques de la UPC-BarcelonaTech (IMTech) )}, Universitat Polit\`{e}cnica de Catalunya $\&$ Centre de Recerca Matemàtica, Barcelona, Spain \\ and
\\ IMCCE, CNRS-UMR8028, Observatoire de Paris, PSL University, Sorbonne
Universit\'{e}, 77 Avenue Denfert-Rochereau,
75014 Paris, France}
\email{eva.miranda@upc.edu}

\begin{document}

\begin{abstract} In this article we consider integrable systems on  manifolds endowed with singular symplectic structures of order one. These structures are symplectic away from an hypersurface  where the symplectic volume  goes either to infinity or to zero in a transversal way (singularity of order one) resulting either in  a $b$-symplectic form or  a folded symplectic form. The hypersurface where the form degenerates is called critical set. We give a new impulse to the investigation of the existence of action-angle coordinates for these structures initiated in \cite{KM} and \cite{KMS} by proving an action-angle theorem for folded symplectic integrable systems.  Contrary to expectations, the action-angle coordinate theorem for folded symplectic manifolds cannot be presented as a cotangent lift as done for symplectic and $b$-symplectic forms in \cite{KM}.  Global constructions of integrable systems are provided and  obstructions for global existence of action-angle coordinates  are investigated in both scenarios. The new topological obstructions found emanate from the topology of the critical set $Z$ of the singular symplectic manifold. The existence of these obstructions in turn implies the existence of singularities for the integrable system on $Z$.

\end{abstract}
\maketitle

\section{Introduction}

  In this article we investigate the integrability of Hamiltonian systems on manifolds endowed with a smooth $2$-form which is symplectic away from an hypersurface $Z$ (called the \emph{critical set}) and  which degenerates in a controlled way (of order one) along it. Either this form lowers its rank at $Z$ and it induces a form on $Z$ with maximal rank or its associated symplectic volume blows-up with a singularity of order one. The manifolds endowed with the first type of singular structure are  called \emph{folded symplectic manifolds} and the ones endowed with the second one are called $b$-symplectic forms. Folded symplectic manifolds  can be thought as symplectic manifolds \emph{with a fold, $Z$} that "mirrors"  the symplectic structure on both sides. The study of \emph{folded symplectic manifolds} complements that of their \emph{"duals"} to $b$-symplectic manifolds which have been largely investigated since \cite{GMP1} and \cite{GL} and are better described as Poisson manifolds whose Poisson bracket looses rank along an hypersurface keeping some transversality properties. This article is also an invitation to consider more degenerate cases (higher order singularities) which will be studied elsewhere and the models provided here can be considered as a \emph{toy model} for more complicated singularities.
{Integrable systems on singular manifolds show up naturally, for instance, in the study of the Toda systems when the particles in interaction collide or are far-away. Singular symplectic manifolds naturally model symplectic manifolds with boundary and, as such, the notion of integrable system is naturally extended to manifolds with boundary.}

  The research of integrability of Hamiltonian systems on these manifolds is of interest both from a Poisson and symplectic point of view.
   The existence of action-angle coordinates on symplectic manifolds has been of major importance as, other than integrating the system itself, it provides semilocal normal forms for integrable Hamiltonian systems which allow, for instance, a deep understanding perturbation theory of these systems (KAM theory). The existence of action-angle coordinates of integrable systems is also useful for quantization as already observed by Einstein when reformulating the Bohr-Sommerfeld quantization conditions \cite{einstein}.

Integrable systems on these singular symplectic manifolds define natural Lagrangian foliations on them and thus naturally yield real polarizations on these manifolds. In particular they are of interest to study geometric quantization of symplectic manifolds with boundary as one of the sources of examples for these singular structures. On symplectic manifolds with boundary deformation quantization is already well-understood \cite{NestandTsygan} and formal geometric quantization has been object of recent study in \cite{weitsman} for non-compact manifolds and in \cite{bquantization1, bmquantization} and \cite{bquantization2} for $b^m$-symplectic manifolds. More specifically, the existence of action-angle coordinates for these structures provides a \emph{primitive} first model for geometric quantization by counting the  integer fibers of the integrable system. As proved in \cite{hamilton, guilleminsternberg, hamiltonmiranda, mps}  this model has been tested to be successful in geometric quantization of toric symplectic manifolds  and refines  the idea of Bohr-Sommerfeld quantization. Understanding action-angle coordinates for integrable systems on singular symplectic manifolds can be a first good step in the study of geometric quantization of singular symplectic manifolds. {Action-angle manifolds on singular symplectic manifolds also provide  natural cotangent-type models that can be useful in understanding the notion of quantum integrable systems (\cite{quantum}, \cite{BBR}) in the singular set-up.}

  The study of \emph{folded symplectic manifolds} comprises the case of origami manifolds \cite{anarita} where additional conditions are imposed on the critical set and a natural global toric action exists. Origami manifolds inherit their denomination from origami paper templates where a superposition of Delzant polytopes \cite{Delzant} gives rise to a toric action on a class of folded symplectic manifolds. Symplectic origami provides an example of integrable system on folded symplectic manifolds but there are other examples motivated by physical systems such as the folded spherical pendulum or the Toda system where the interacting particles are far-away.

In this article we show the existence of $b$-integrable systems on  $b$-symplectic manifold of dimension $4$ having as critical set a Seifert manifold, and via the desingularization technique we obtain folded integrable systems on the associated desingularized folded symplectic manifold.
We prove the existence of action-angle coordinates à la Liouville-Mineur-Arnold exploring the Hamiltonian actions by tori on folded symplectic manifolds. The action variable are not exactly coordinates, since these variables can degenerate in a certain way.  We show that this action-angle theorem cannot always be interpreted in terms of a cotangent model as in the symplectic and $b$-symplectic case.

  We end up this article investigating the obstruction theory of  global existence of action-angle coordinates exhibiting a new topological obstruction for the singular symplectic manifolds that lives on the critical set of the singular symplectic form. This yields examples of integrable systems on $b$-symplectic manifolds and folded symplectic manifolds with critical set non diffeomorphic to a product of a symplectic leaf with a circle. For those systems the toric action does not even extend to a neighborhood of the critical set.
We end up this article observing that the existence of finite isotropy for the transverse $S^1$-action given by the modular vector field obstructs the uniformization of periods of the associated torus action  on the $b$-symplectic manifold and automatically yields the existence of singularities of the integrable system on the critical locus of the $b$-symplectic structure.

 \textbf{Organization of this article:} In Section 2 we introduce the basic tools in $b$-symplectic and folded symplectic geometry. In Section 3 we investigate Hamiltonian dynamics on folded symplectic manifolds and introduce folded integrable systems. In Section 4 we provide a list of motivating examples for integrable systems on folded symplectic manifolds. In Section 5 we prove an action-angle theorem (Theorem \ref{thm:aa}) for folded symplectic manifolds and discuss folded cotangent models. We end up this section discussing why this theorem does not always yield equivalence with a cotangent model. Section 6 contains constructions of integrable systems on  4-dimensional $b$-symplectic manifolds having as critical set a Seifert manifold (Theorem \ref{thm:bexist}) and on any folded symplectic manifold which desingularizes it. Section 7 investigates the existence of global action-angle coordinates and highlights the non-triviality of the mapping torus as topological obstruction to global existence of action-angle coordinates (Theorems \ref{thm:globalaa1} and \ref{thm:globalaa2}).

\section{Preliminaries}

In this article we consider forms $\omega$ on even dimensional manifolds $M^{2n}$ which are symplectic away from a hypersurface $Z$ and such that $\omega^n$ either cuts the zero section of the bundle $\Lambda^n(T^*M)$ transversally or goes to infinity in a controlled way along $Z$. So, in particular  $\omega^n$  defines a volume form away from $Z$.

For the class of forms for which  $\omega^n$ either cuts the zero section of the bundle $\Lambda^n(T^*M)$ transversally we require an extra condition to guarantee maximal rank (see below). These forms are called \emph{folded symplectic forms} as they can be visualized as symplectic manifolds which are folded along the folding hypersurface.
\subsection{Basics on folded symplectic manifolds}
{ We recall here some basic facts on folded symplectic manifolds.}
\begin{defi}
Let $M$ be a $2n$-dimensional manifold. We say that $\omega \in \Omega^2(M)$ is folded-symplectic if
\begin{enumerate}
    \item $d\omega=0$,
    \item $\omega^n \pitchfork \mathcal{O}$, where $\mathcal{O}$ is the zero section of $\bigwedge^{2n}(T^*M)$, hence $Z=\{ p\in M, \omega^n(p)=0\}$ is a codimension $1$ submanifold,
    \item $i: Z \rightarrow M$ is the inclusion map, ${i}^*\omega$ has maximal rank $2n-2$.
\end{enumerate}
We say that $(M,\omega)$ is a \textbf{folded-symplectic manifold} and we call $Z \subset M$ the folding hypersurface.
\end{defi}

The following theorem is an analog of Darboux theorem for folded symplectic
forms~\cite{M}:

\begin{theorem}[Martinet]\label{thm:martinet}
  For any point $p$ on the folding hypersurface $Z$ of a folded symplectic manifold $(M^{2n}, \omega)$ there is a local system of coordinates $(x_1, y_1, \dots, x_n, y_n)$ centered at $p$  such that $Z$ is locally given by $x_1=0$ and   $$\omega= x_1 dx_1 \wedge dy_1 + dx_2 \wedge dy_2 + \ldots + dx_n \wedge dy_n. $$

\end{theorem}

   Let $(M,\omega)$ be a $2n$-dimensional folded symplectic manifold.
Let $i : Z \hookrightarrow M$ be the inclusion of the
folding hypersurface $Z$. The induced restriction $i ^* \omega$ has a one-dimensional kernel at
each point. We denote by $V$ the bundle $\ker i^*\omega$ defined at $Z$, and $L= V\cap TZ$ the null line bundle.

The following theorem ~\cite[Theorem 1]{CGW} is  a Moser type theorem for folded symplectic manifolds which extends the local normal form above to a neighborhood of $Z$.
For the null line bundle we consider $\alpha$ a one-form such that $\alpha(v)=1$ for a non-vanishing section $v$ of $L$.

\begin{theorem}\label{thm:moser}
Suppose that Z is compact. Then there exists a neighborhood $U$
of $Z$ and an orientation preserving diffeomorphism,
$$ \varphi:Z\times(-\varepsilon,\varepsilon)\to U$$
where $\varepsilon>0$ and $U$
such that $\varphi(x,0)=x$ for all $x\in Z$ and
$$\varphi^* \omega = p^* i^* \omega
   + d \left( t^2 p^* \alpha \right) \ ,$$
with $p: Z \times (-\varepsilon,\varepsilon) \to Z$
the projection onto the first factor, the map
$i : Z \hookrightarrow M$ the inclusion and
$t$ the real coordinate on the interval $(-\varepsilon,\varepsilon)$.
\end{theorem}

A special class of folded-symplectic manifolds are origami symplectic manifolds
for which the null foliation $L$ defines a fibration.

\begin{defi} An \textbf{origami manifold} is a folded symplectic manifold $(M, \omega)$ whose null foliation
is fibrating with oriented circle fibers, $\pi$, over a compact \emph{base}, $B$.

{
\begin{center}
\begin{tikzcd}
S^1 \arrow[hookrightarrow]{r}{} & \arrow{d}{\pi}  Z  \\%
{} & B
\end{tikzcd}
\end{center}
}

The form $\omega$ is called an \textbf{origami form}
and the null foliation
is called the \textbf{null fibration}.
\end{defi}
\begin{rem}
On an origami manifold,
the base $B$ is naturally symplectic with
symplectic form $\omega_B$ on $B$ satisfying
$i^* \omega = \pi^* \omega_B.$
\end{rem}

\begin{ex}
\label{ex:spheres}
Consider the unit sphere $S^{2n}\subset\mathbb R^{2n+1}$ given by equation $\sum_{i^1}^{n} (x_i^2+y_i^2)+ z^2=1$
with global coordinates $x_1,y_1,\ldots,x_n,y_n,z$ on $\mathbb R^{2n+1}$.
Let $\omega_0$ be the restriction to $S^{2n}$ of the form
$dx_1 \wedge dy_1 + \ldots + dx_n \wedge dy_n$ which in polar coordinates reads
$r_1 dr_1 \wedge d\theta_1 + \ldots + r_n dr_n \wedge d\theta_n$.
Then $\omega_0$ is a folded symplectic form.
The folding hypersurface is the sphere given by the
intersection with the plane $z=0$.
It is easy to check that the null foliation is the Hopf foliation (see for instance \cite{anarita}).
\end{ex}

\subsection{Basics  on $b$-symplectic manifolds}

In this section we give a crash course on {\bf $b$-symplectic/Poisson manifolds}.

The study of $b$-symplectic manifolds starts with a similar definition to that of folded symplectic manifolds but in the context of Poisson geometry.
Given a symplectic form $\omega$ we can naturally associate a Poisson bracket to any pair of smooth functions $f,g\in \mathcal{C}^\infty(M^{2n})$ from the symplectic structure as follows
$$\{f,g\}=\omega(X_f, X_g)$$

\noindent where the vector fields $X_f$ and $X_g$ stand for the Hamiltonian vector fields with respect to $\omega$.

From the equation above it is simple to check that $\{f,g\}=X_f(g)$ so the bracket defines a biderivation (Leibnitz rule), it is antisymmetric and because $X_{\{f,g\}}=[X_f, X_g]$, it also satisfies the Jacobi identity (i.e., $\{ \{f,g\}, h \}+\{ \{g, h\}, f \}+\{ \{h,f\}, g \}=0$ for any triple of smooth functions $f,g$ and $h$.

A general Poisson structure is defined as a general  antisymmetric bracket on any manifold (not necessarily even dimensional) $\{\cdot,\cdot\}: \mathcal{C}^\infty(M)\times \mathcal{C}^\infty(M)\longrightarrow \mathcal{C}^\infty(M)$ satisfying Leibnitz rules and Jacobi identity.

Because a Poisson bracket defines a biderivation, we can write it as a bivector field $\Pi\in \Gamma(\Lambda^2(TM))$. The correspondence between Poisson brackets and Poisson bivector fields is clarified by the equation $$\Pi(df,dg)=\{f,g\}.$$

The Jacobi identify defines an additional constraint and not every bivector field defines a Poisson structure. Bivector fields which are Poisson satisfy the integrability equation
$[\Pi, \Pi]=0$ where the bracket is the Schouten bracket, the natural extension of the Lie bracket to bivector fields.

In total analogy with the symplectic case, given a function we may define the Hamiltonian vector field via the equation:
$X_f:=\Pi(df,\cdot)$.
Observe that, in particular the equation $\{f,g\}=X_f(g)$ also holds in the general Poisson context.

Let us now consider a Poisson bivector field on an even dimensional manifold which is symplectic away from an hypersurface. When these Poisson bivector fields fulfill transversality conditions along the hypersurface many techniques from the symplectic realm can be exported to study this class of Poisson structures. These are called $b$-Poisson manifolds and  have been studied and analyzed in detail starting in \cite{GMP1}.
\begin{defi}\label{def:bpoisson}
Let $(M^{2n},\Pi)$ be a Poisson manifold. If the map
$$p\in M\mapsto(\Pi(p))^n\in\bigwedge^{2n}(T(M^{2n}))$$
is transverse to the zero section, then $\Pi$ is called a \textbf{$b$-Poisson structure} on $M$. The pair $(M^{2n},\Pi)$ is called a \textbf{$b$-Poisson manifold}. The vanishing set of $\Pi^n$ is a hypersurface denoted by $Z$ and called the {\bf critical hypersurface} of $(M^{2n},\Pi)$.
\end{defi}

A list of examples can be found and analyzed in detail in \cite{arnaueva}. In the particular case of surfaces, these structures coincide with the stable Poisson structures classified by \cite{radko}.
The next example is the prototypical \emph{Radko sphere}.

\begin{ex}
 We endow the  $2$-sphere $\mathbb S^2$ with the coordinates $(h,\theta)$, where $h$ denotes the height function and $\theta$ is the  angle. The Poisson structure written as $\Pi=h\frac{\partial}{\partial h}\wedge\frac{\partial}{\partial \theta}$ vanishes transversally along the equator $Z=\{h=0\}$ and thus it defines a $b$-Poisson structure on the pair $(\mathbb S^2,Z)$.
\end{ex}

The product of  two $b$-Poisson manifolds is not a $b$-Poisson manifold but the product of a $b$-Poisson surface with a symplectic manifold is a $b$-Poisson manifold as described in the example below.

\begin{ex}\label{example2}
For higher dimensions we may consider the following product structures: let $(\mathbb S^2,Z)$ be the sphere in the example above and $(S^{2n},\pi_S)$ be a symplectic manifold, then $(\mathbb S^2\times S,\pi_{\mathbb S^2}+\pi_S)$ is a $b$-Poisson manifold of dimension  $2n+2$.  We may replace $(\mathbb S^2,Z)$ by any Radko compact surface $(R,\pi_R)$ (see for instance \cite{GMP1}).

\end{ex}

Other examples come from foliation theory and from the theory of cosymplectic manifolds:

\begin{ex}\label{example3}
Let $(N^{2n+1},\pi)$ be a Poisson manifold of constant corank $1$. Let us assume that there exists a vector field $X$ which is a  Poisson vector field and let $f:\mathbb S^1\to\mathbb R$ be a smooth function. The bivector field given by
$$\Pi=f(\theta)\frac{\partial}{\partial\theta}\wedge X+\pi$$
defines a $b$-Poisson structure on the product $S^1\times N$ whenever the function $f$ vanishes linearly and the vector field $X$ is transverse to the symplectic leaves of $N^{2n+1}$. In this case, the critical hypersurface is formed by the union of as many copies of $N$ as zeros of $f$.

\end{ex}

The  example above is generic in the sense that any $b$-Poisson structure can be described in this way in a neighborhood of a critical hypersurface $N$. The critical hypersurface $N$ has a natural cosymplectic structure associated to it. In particular, this example realizes a given cosymplectic manifold as a connected component of a cosymplectic manifold which is a critical set of a $b$-Poisson manifold.
This is the content of example 19 in \cite{GMP1}.

Around any point in $Z$, the $b$-Darboux theorem (see \cite{GMP1} and \cite{NestandTsygan}) guarantees that it is always possible to find local coordinates with respect to which the $b$-Poisson structure as stated below:

\begin{theorem}[$b$-Darboux] For any point $p\in Z$ on the critical hypersurface of a $b$-Poisson manifold we may find local coordinates centered at $p$ for which the $b$-Poisson structure $\Pi$
can be written as:
\[
\Pi= \sum_{i=1}^{n-1} \frac{\partial}{\partial x_i} \wedge \frac{\partial}{\partial y_i} + t \frac{\partial}{\partial t} \wedge \frac{\partial}{\partial z}.
\]

\end{theorem}

Thus $b$-Poisson manifolds and symplectic manifolds have many things in common. Indeed
it is possible to work with the language of forms by admitting $\frac{df}{f}$ where $f$ is the defining function of $Z$ as a \emph{legal} form of an extended complex.
This is the complex of $b$-forms originally introduced by Richard Melrose \cite{Melrose} to study the index theorem on manifolds with boundary.

In order to introduce this language properly we briefly recall the construction of $b$-forms: Given a pair $(M,Z)$ where $Z$ is an hypersurface a $b$-vector field is a vector field on $M$ tangent to $Z$. The space of $b$-vector fields can be naturally identified as sections of a vector bundle on $M$ called the $b$-tangent bundle $^bTM$. When we refer to $b$-forms we consider sections of the exterior algebra of the its dual, the $b$-cotangent bundle $^bT^*M := (^bTM)^*$.

Any $b$-form of degree $k$ can be written as $\omega= \frac{df}{f}\wedge \alpha+\beta$ where  $\alpha$ and $\beta$ are $k-1$ and $k$ smooth De Rham forms respectively and $f$ is a defining function for $Z$.

\begin{defi}({\bf $b$-functions})
 A $b$-function on a $b$-manifold $(M,Z)$ is a function which is smooth away from the critical set $Z$, and near $Z$ has the form
$$ c \log |t| + g, $$
where $c\in \R, g\in C^\infty,$ and $t$ is a local defining function. The sheaf of $b$-functions is denoted $^b \! C^\infty$.
\end{defi}

A closed $b$-form of degree $2$ which is nondegenerate as a section of the bundle $\Lambda^2 ({^b}T^*M)$ is called a {\textbf{ $b$-symplectic form}}. As it is proved in \cite{GMP1} there is a one-to-one correspondence between $b$-symplectic forms and $b$-Poisson forms. In particular we may re-state the Darboux normal form in the language of $b$-forms as done below

\begin{theorem}\label{theorem:Darboux2}\textbf{($b$-Darboux theorem)}
Let $\omega$ be a $b$-symplectic form on $(M,Z)$ and $p\in Z$. Then we can find a coordinate chart $(U,x_1,y_1,\ldots,x_n,y_n)$ centered at $p$ such that on $U$ the hypersurface $Z$ is locally defined by $y_1=0$ and
$$\omega=d x_1\wedge\frac{d y_1}{y_1}+\sum_{i=2}^n d x_i\wedge d y_i.$$
\end{theorem}

 For any $b$-function on a $b$-symplectic manifold $(M, \omega)$ the \textbf{$b$-Hamiltonian} vector field is the  one  $X_{f_n}$ defined by $\iota_{X_{f_n}}\omega = -df_n$.

  A $\mathbb{T}^k$ action on a $b$-symplectic manifold $(M^{2n}, \omega)$ is called \textbf{$b$-Hamiltonian} if the fundamental vector fields are the $b$-Hamiltonian vector fields of functions which Poisson commute. Such an action is called \textbf{toric} if $k = n$.

The critical hypersurface $Z$ of a $b$-symplectic structure has an induced regular Poisson structure which can also be visualized as a cosymplectic manifold (see \cite{GMP1, GMP2}).

In \cite{GMP2} it was shown that if $Z$ is compact and connected, then the critical set $Z$  is the mapping torus of any of its symplectic leaves $L$ by the flow of the any choice of modular vector field $u$:
\[
 Z = (L \times [0,k])/_{(x,0)\sim (\phi(x),k)},
\]
where $k$ is a certain positive real number and $\phi$ is the time-$k$ flow of $u$. In particular, all the symplectic leaves inside $Z$ are symplectomorphic. { As in \cite{GMP1}, we refer to a fixed symplectomorphism inducing the mapping torus as the \textbf{monodromy} of $Z$.}

This yields  the following definition:

\begin{defi}[Modular period]\label{def:modperiod}
 Taking any modular vector field $u_{mod}^\Omega$, the {\bf modular period} of $Z$ is the number $k$ such that $Z$ is the mapping torus
$$ Z = (L \times [0,k])/_{(x,0)\sim (\phi(x),k)},$$
and the time-$t$ flow of  $u_{mod}^\Omega$ is translation by $t$ in the $[0, k]$ factor above.
\end{defi}

\subsubsection{The (twisted) $b$-cotangent lift}\label{ssec:bcotangentlift}
The cotangent lift can also be defined on the  $b$-cotangent bundle of a smooth manifold. In this case there are two different $1$-forms that provide the same geometrical structure on the $b$-cotangent bundle (a $b$-symplectic form). These are the canonical (Liouville) $b$-form  and the twisted $b$-form. Both forms  of degree $1$ have the same differential ( a smooth $b$-symplectic form) but are indeed non-smooth forms. The $b$-cotangent lift in each of the cases is defined in a different manner.
These were studied in detail in \cite{KM}. In this article we focus on the twisted $b$-cotangent lift as it gives the right model for the structure of a $b$-integrable system.
\begin{defi}
Let $T^* \mathbb T^n$ be endowed with the standard coordinates $(\theta, a)$, $\theta \in \mathbb T^n$, $a \in \mathbb R^n$ and consider again the action on $T^* \mathbb T^n$ induced by lifting translations of the torus $\mathbb T^n$. Define the following  non-smooth one-form away from the hypersurface
$Z=\{a_1 = 0\}$~:
$$\lambda_{tw, c} \log|a_1| d \theta_1 + \sum_{i=2}^n a_i d\theta_i.$$
Then, the form $\omega:=-d\lambda_{tw, c}$ is a $b$-symplectic form on $T^* \mathbb T^n$, called the \textit{twisted $b$-symplectic form} on $T^* \mathbb T^n$. In coordinates:
\begin{equation}
\omega_{tw, c}:=\frac{c}{a_1} d\theta_1\wedge d a_1 + \sum_{i=2}^n d\theta_i\wedge da_i.
\label{eq:twistedbform}
\end{equation}
\label{def:twistedbform}
\end{defi}

Observe that this twisted forms comes endowed with a local invariant: The constant $c$. The interpretation of this invariant is that this gives the period of the modular vector field.

We call the lift together with the $b$-symplectic form \eqref{eq:twistedbform} the \textbf{twisted $b$-cotangent lift} with modular period $c$ on the cotangent space of a torus.

As it was deeply studied in \cite{KM} the lifted action can be extended to groups of type $S^1\times H$ which turns out to be $b$-Hamiltonian in general.

%

%
%

\subsubsection{$b$-Integrable systems}
A $b$-symplectic manifold/$b$-Poisson manifold can be seen as a standard Poisson manifold. Along the critical set $Z$ the Liouville tori determined by a standard integrable system have dimension $n-1$ which is not convenient to model integrable systems on $b$-symplectic manifolds where the critical set $Z$ represent the direction of \emph{infinity} in celestial mechanics and we would expect to have $n$-dimensional tori.

This is why in this context it is more natural to talk abou \emph{$b$-integrable system} as follows:
 \begin{defi}\label{def:pbintegrable}  A  \textbf{$b$-integrable system} on a $2n$-dimensional $b$-symplectic manifold $(M^{2n},\omega)$ is a set of $n$  $b$-functions which are pairwise Poisson commuting  $F=(f_1,\ldots,f_{n-1},f_n)$   with
$df_1 \wedge \dots \wedge df_n\neq 0$  as a section of $\wedge^n  ({^b} T^*(M))$ on a dense subset of $M$ and on a dense subset of $Z$. A point in $M$ is {\bf regular} if the vector fields $X_{f_1}, \dots, X_{f_n}$ are linearly independent (as \emph{smooth} vector fields) at it. \end{defi}

For these systems  an action-angle coordinate, proved in \cite{KMS},  shows the existence of a semilocal invariant in the neighbourhood of $Z$ (the modular period):

\begin{theorem}\label{thm:aa}
 Let $(M, \omega, F = (f_1, \dots, f_{n-1}, f_n = \log|t|))$ be a $b$-integrable system, and let $m \in Z$ be a regular point for which the integral manifold containing $m$ is compact, i.e. a Liouville torus $F_m$. Then there exists an open neighborhood $U$ of the torus $F_m$ and coordinates $(\theta_1,\dots,\theta_n,\sigma_1,\dots,\sigma_{n}): U \to  \mathbb T^n \times B^n$ such that
\begin{equation}
        \omega|_U =\sum_{i=1}^{n-1} d\sigma_i \wedge d\theta_i  + \frac{c}{\sigma_n} d\sigma_n \wedge d\theta_n,
\end{equation}
where the coordinates $\sigma_1,\dots,\sigma_n$ depend only on $F$ and the number $c$ is the modular period of the component of $Z$ containing $m$.
\end{theorem}

 In \cite{KM} this normal form was identified as a cotangent model:

\begin{theorem} Let $F=(f_1,\ldots,f_n)$ be a $b$-integrable system on the $b$-symplectic manifold $(M,\omega)$.  Then semilocally around a regular Liouville torus $\mathbb T$, which lies inside the exceptional hypersurface $Z$ of $M$, the system is equivalent to the cotangent model $(T^\ast \T^n)_{tw, c} $ restricted to a neighbourhood of $(T^\ast \mathbb T^n)_0$. Here $c$ is the modular period of the connected component of $Z$ containing  $\mathbb T$.
\end{theorem}

\subsection{$b$-symplectic manifolds and folded symplectic manifolds as duals}

In \cite{Deblog} the following theorem is proved in the more general setting of $b^m$-symplectic structures with singularities of higher order:

\begin{theorem}\label{thm:desing}
{ Let $\omega$ be a} $b$-symplectic structure  on a compact manifold $M$  and let $Z$ be its critical hypersurface.
There exists  a family of folded symplectic forms ${\omega_{\epsilon}}$ which coincide with  the $b$-symplectic form
    $\omega$ outside an $\epsilon$-neighborhood of $Z$.

\end{theorem}

As a consequence of this result any $b$-symplectic manifold admits a folded symplectic structure. However, it is well-known that the converse statement does not hold
as not every folded symplectic form can be presented as a desingularization of a $b$-symplectic structures. In particular, any compact orientable $4$-dimensional  manifold admits
a folded symplectic form \cite{anacannas} but not every $4$-dimensional compact manifold admits a $b$-symplectic manifold. For instance the 4-sphere $S^4$ does not admit a $b$-symplectic structure as it was proven in \cite{GMP1} that the  class determined by the $b$-symplectic form is non-vanishing.

As we will see in the next section not every folded integrable system on a folded symplectic manifold can be obtained via desingularization from a $b$-integrable system.

\section{Hamiltonian dynamics on folded symplectic manifolds}

 Let $(M,\omega)$ be a folded symplectic manifold, with folding hypersurface $Z$. Consider $p$ a point in $Z$, applying Theorem \ref{thm:martinet}  the folded-symplectic form $\omega$ can be written in a neighborhood of $p$ as:
$$ \omega= t dt \wedge dq +\sum_{i=2}^n dx_i \wedge dy_i. $$
\noindent with $t=0$ defining the folding hypersurface.
 The singularity in $\omega$ prevents the Hamiltonian equation $\iota_X\omega=-df$ from having a solution for every possible function $f$. So  not every function $f \in C^{\infty}(U) $ defines locally a Hamiltonian vector field.
\begin{ex}
\normalfont
    Let $(U;t,q,...,x_n,y_n)$ be a chart where $\omega$ takes the folded-Darboux form mentioned above. Take for example the function $f=t$. By imposing the Hamiltonian equation we get that $X=\frac{1}{t}\pp{}{q}$, which is not a well defined smooth vector field.
\end{ex}
Fortunately, we can characterize the set of functions which define smooth Hamiltonian vector fields.
\begin{lemma}\label{lem:foldfunc}
A function $f:M \rightarrow \R$ in a folded symplectic manifold $(M,\omega)$ has an associated smooth  Hamiltonian vector field $X_f$ if and only if $df|_Z(v)=0$ for every $v\in V$. Furthermore $X_f$ is tangent to $Z$.

\end{lemma}

\begin{proof}
Assume that $f$ has a well-defined smooth  Hamiltonian vector field at a point $p$ in $Z$. Take Darboux coordinates $(t,q,...,x_n,y_n)$ at a neighborhood $U$ of $p$. In these coordinates, the form can be written as:
$$ \omega= t dt \wedge dq + \sum_{i=2}^n dx_i \wedge dy_i. $$
Any vector field can be written as
$ X= a_1 \pp{}{t} + b_1 \pp {}{q} + ... + a_n \pp{}{x_n} + b_n\pp{}{y_n}$.
Imposing the Hamiltonian equation $\iota_X \omega= -df$ we obtain that
\begin{equation*}
\begin{cases}
    a_1 &= -\pp{f}{q} \frac{1}{t} \\
    b_1 &= \pp{f}{t} \frac{1}{t} \\
    a_i &= -\pp{f}{q}, \enspace i=2,...,n \\
    b_i &= \pp{f}{t}, \enspace i=2,...,n.
\end{cases}
\end{equation*}
The coefficients $a_1$ and $b_1$ well defined, we need that $-\pp{f}{q} \frac{1}{t}$ and $\pp{f}{t}\frac{1}{t}$ to be smooth. For this to hold, we need that $\pp{f}{q}=t H$ and $\pp{f}{t}=t F$ for some smooth function $H$ and $G$. From the second equation we get that $f=t^2f_1+ f_2(q,x_2,...,y_n)$ for some smooth functions $f_1,f_2$. The first equation implies that $\pp{f_2}{q}=0$. Hence $f$ has the form
\begin{equation} \label{foldfunc}
 f=t^2f_1 + f_2(x_2,...,y_n),
\end{equation}
which implies that $df(\pp{}{t})|_Z=df(\pp{}{q})|_Z=0$ since $Z=\{t=0\}$. Since $\pp{}{t}$ and $\pp{}{q}$ space $V=\ker \omega|_Z$, we get that $df(v)=0$ for each $v\in \ker \omega|_Z$. The converse is obviously true: if $df(v)=0$ for each $v\in \ker \omega |_Z$ then it defines a smooth Hamiltonian vector field.

It follows from equation \eqref{foldfunc} that $\pp{f}{q}=t^2 \pp{f_1}{q}$ which implies that $a_1|_{\{t=0\}}=-t\pp{f_1}{q}|_{\{t=0\}}=0$. We deduce that $X_f$, the Hamiltonian vector field of $f$, is tangent to $Z$.
\end{proof}
We denote these functions as folded functions.
\begin{defi}
A function $f:M \rightarrow \R$ in a folded symplectic manifold $(M,\omega)$ is a \textbf{folded function} if $df|_Z(v)=0$ for every $v\in V=\ker \omega|_Z$.
\end{defi}

Note that even if a Hamiltonian vector field $X_f$ is always tangent to $Z$, one can obtain non vanishing components of $X_f$ in the null line bundle $L$. If one takes $n$ folded functions, we will always have $df_1 \wedge ... \wedge df_n|_Z=0$ when we look it as a section of $\Lambda^nT^*M$. However, the $n$ functions can define $n$ linearly independent Hamiltonian vector fields. This yields the following definition:
\begin{defi}
An integrable system on a folded symplectic manifold $(M,\omega)$ with critical surface $Z$ is a set of functions $F=(f_1,...,f_n)$ such that they define Hamiltonian vector fields which are independent on a dense set of $Z$ and $M$, and commute with respect to $\omega$.

\end{defi}

Around the regular points of the integrable system, the expression of the functions can be simplified and as a consequence the Poisson bracket of the functions is well-defined:

\begin{lemma}\label{normalint}
Near a regular point of an integrable system, there exist coordinates $(t,q,x_2,...,y_n)$ such that $\omega = tdt\wedge dq + \sum_{i=2}^n dx_i \wedge dy_i$, and the integrable system has the form
\begin{align*}
    f_1 &= t^2/2 \\
    f_2 &= g_2(t,q,x_2,...,y_n) {t}^{k_2} + h_2(x_2,y_2,...,x_n,y_n) \\
    &\vdots \\
    f_n &= g_n(t,q,x_2,...,y_n) {t}^{k_n} + h_n(x_2,y_2,...,x_n,y_n),
\end{align*}
for $k_2, \dots, k_n\in \mathbb N$ all of them $\geq 2$
and $t$ is a defining function of $Z$.
\end{lemma}

\begin{proof}
 Denote the inclusion of $Z$ in $M$ by $i:Z \hookrightarrow M$. Since  the pullback to $Z$ of the folded symplectic form $i^*\omega$ has rank $2n-2$, there are at most $n-1$ independent Hamiltonian vector fields tangent to $Z$  such that  $\langle X_1,...,X_{n-1} \rangle$ has no component in $\ker i^*\omega$. This implies that at any regular point $p \in Z$ of an integrable system, one of the $n$ independent Hamiltonian vector fields $X_1,...,X_n$ has a component in $\ker i^*\omega$. We might assume it is the first one $X_1$.

Let us show that in the points close to $p$ in $Z$, this vector field $X_1$ can be written as $X_1=v+X'$, where $v\in \ker i^*w$ and $X' \in \langle X_2,...,X_n \rangle$.

Indeed if $X_1$ had a component in the complement of $\ker i^*\omega \cup \langle X_2,...,X_n \rangle$, it would have a component either in the symplectic orthogonal space to $\langle X_2,...,X_n \rangle$ with respect to $i^*\omega$ or in $TZ^\perp$. However, we know that the Hamiltonian vector fields with respect to $\omega$ are tangent to $Z$, and so $X$ would have a component in the symplectic orthogonal to $\langle X_2,...,X_n \rangle$ and would not  satisfy $\omega(X,X_i)=0$. In particular, we can take a new basis of Hamiltonian vector fields generating $\langle X_1,...,X_n \rangle$ in a neighborhood of $p$ such that $X_1$ lies exactly in the kernel of $i^*\omega$ in $U\cap Z$.

Take local coordinates in a neighborhood $U$ of $p$ such that $X_1=\frac{\partial}{\partial q}$. Take symplectic coordinates $(x_2,y_2,...,x_n,y_n)$ of $i^*\omega$, the existence of such coordinates follows from the Darboux theorem for closed two forms of constant rank \cite[Proposition 13.7]{marle}. We can now use Theorem \ref{thm:moser} with $\alpha=dq$ to conclude that
$$ \omega= tdt\wedge dq + \sum_{i=2}^n dx_i \wedge dy_i. $$

In these coordinates the vector field $X_1$ is the Hamiltonian vector field of $t^2/2$, hence $f_1=t^2/2$. The remaining functions $f_2,...,f_n$ are folded functions and can be expressed as in Equation \eqref{foldfunc}. This concludes the proof of the lemma.
\end{proof}

\subsection{Folded cotangent bundle}
In this subsection we recall the construction of \cite{H},  a dual of the $b$-cotangent bundle for folded symplectic manifolds.

\begin{defi}
Let $M$ a manifold and $Z$ a closed hypersurface. Let $V$ a rank $1$ subbundle of $i_Z^*TM$ so that for all $p\in Z$ the fiber $V_p$ is transverse to $T_pZ$. We define for each open subset $U\subset M$
$$ \Omega_V^1(U):=\{ \alpha \in \Omega^1(U)|\enspace \alpha|_V=0 \}, $$
the space of $1$-forms on $U$ vanishing on $V$. If $U\cap Z= \emptyset$ then it is just $\Omega^1(U)$.
\end{defi}

Following \cite{H} there exists a vector bundle $T_V^*M$ called the \textbf{folded cotangent bundle}, of rank $n$ whose global sections are isomorphic to $\Omega^1_V(M)$. This vector bundle is unique up to isomorphism, independently of the chosen $V$. For a small open neighborhood $U$ of a point in $Z$, there exist suitable coordinates $(x_1,...,x_{n-1})$ in $U\cap Z$ and a coordinate $t$ such that $(x_1,...,x_{n-1},t)$ are coordinates in $U$ and $T_V^*(U)$ is generated by $dx_1,...,dx_{n-1},tdt$.  The dual bundle to $T_V^*M$ is denoted by $T_VM$ and called the folded tangent bundle.

In this bundle there is a canonical folded symplectic form which is obtained by taking a Liouville form $\lambda_f$ which is canonical as it satisfies the \emph{Liouville-type} equation $\langle \lambda_f|_p, v \rangle = \langle p, (\pi_p)_*(v)\rangle$
for every $v \in T_V (T_V^*M)$ and $p\in T_V^*M$. In coordinates $(x_1,...,x_n,p_1,...,p_n)$ we can write
$$ \lambda= p_1 x_1dx_1 + \sum_{i=2}^n p_i dx_i. $$
Its derivative gives rise to a folded symplectic structure
$$ \omega_f=d\lambda= x_1 dp_1\wedge dx_1 + \sum_{i=2}^n dp_i \wedge dx_i $$
\noindent which looks like the Darboux-type folded symplectic structure.
The introduction of this bundle allows to restate the definition of a folded integrable system in terms of the folded cotangent bundle.

\begin{defi}
An integrable system on a folded symplectic manifold $(M,\omega)$ is a set of folded functions $F=(f_1,...,f_n)$ for which $df_1\wedge ... \wedge df_n \neq 0$  as sections of  {$\Lambda^n T_V^*M$} on a dense set of $M$ and $Z$,  and whose Hamiltonian vector fields commute with respect to $\omega$.

\end{defi}

Even if $\omega$ does not define a Poisson bracket in $Z$ because the Hamiltonian vector fields are not defined for non-folded functions, the bracket is well defined for folded functions and the commutation condition $\omega(X_{f_i},X_{f_j})=0$ for two Hamiltonian vector fields is still well-defined.

%
%

\section{Motivating examples}

In this section we present a series of examples of folded integrable systems. In particular, we exhibit examples of folded integrable systems whose dynamics cannot possibly be modeled by $b$-integrable systems. This will motivate the development of the theory of folded integrable systems, and in particular of the existence of action-angle coordinates.

\subsection{Double collision in two particles system}

In the literature of celestial mechanics like the restricted 3-body problem or the n-body problem several regularization transformation associated to ad-hoc changes (like time reparametrization) bring singularities into the symplectic structure.
Below we describe a model of double collision in two particle systems where McGehee type changes are implemented.
We model a system of two particles under the influence of a potential energy function of the form $U(x)=-|x|^{-\alpha}$, with $\alpha >0$. In the phase space $(x,y)\in \R^2\times \R^2$ it is a Hamiltonian system with Hamiltonian function $F=\frac{1}{2} |y|^2 -|x|^{-\alpha}$. Let us introduce a notation for two constants: denote $\beta= \alpha/2$  and $\gamma=\frac{1}{\beta+1}$. By implementing the change of coordinates:
\begin{equation*}
\begin{cases}
x&= r^{\gamma}e^{i\theta} \\
y&= r^{-\beta \gamma}(v+iw)e^{i\theta}
\end{cases}
\end{equation*}
and scaling with a new time parameter $\tau$ such that $dt=rd\tau$ we obtain the equations of motion
\begin{equation*}
\begin{cases}
r'&= (\beta+1) rv \\
v'&= w^2+\beta(v^2-2) \\
\theta'&= w \\
w' &= (\beta-1)\omega v
\end{cases}.
\end{equation*}

We will model the collision set $\{ r=0 \}$ in the case $\beta=1$ as the folding hypersurface of a folded symplectic manifold endowed with  a folded integrable system. Let us consider the folded symplectic form $\omega= rdr\wedge dv + d\theta \wedge dw$ in the manifold $T^*(\R \times S^1) \cong \R^2 \times S^1 \times \R$ with coordinates $(r,v,\theta,w)$. We take the folded Hamiltonian function
$$ H= -\frac{r^2}{2}(w^2 + (v^2-2)) + \frac{w^2}{2}. $$
Observe that $dH=-r^2
vdv + (w^2 + v^2-2)2r dr + (w-r^2w)dw$, and the Hamiltonian vector field is
$$ X_H= -rv \pp{}{r} + (w^2+v^2-2) \pp{}{v} + (w+ r^2w) \pp{}{\theta} $$
 The equations of motion in the critical hypersurface $\{r=0\}$ coincide with the equations of motion in the collision manifold of the original problem, hence providing a folded Hamiltonian model for it. In fact, even the linear asymptotic behavior close to collision is captured by the model. Observe that $X$ commutes with $\pp{}{\theta}$, which is a Hamiltonian vector field for the function $f_2=w$. Hence the dynamics are modelled by a folded integrable system given by $F=(f_1=H,f_2)$ in $T^*(\R \times S^1)$ with folded symplectic structure $rdr\wedge dv + d\theta \wedge dw$.

%

\subsection{Folded integrable systems on toric origami manifolds}

Not all integrable systems on folded symplectic manifolds come from standard systems on symplectic manifolds after singularization transformations or regularization techniques as in the example above.  Take for instance $\R^4$ with the standard symplectic structure $\omega=dx_1 \wedge dy_1 + dx_2 \wedge dy_2$. The function $f_1=x_1^2+y_1^2+x_2^2+y_2^2$ and $f_2=x_1y_2-x_2y_1$ commute with respect to $\omega$. There is a natural folding map from the sphere $S^4$ to $\bar{D^4}$, that we denote $\pi$. It is a standard fact that $\pi^*\omega$ is a folded symplectic structure in $S^4$, and in fact an origami symplectic structure. Taking $F=(\pi^*f_1,\pi^*f_2)$ yields an example of a folded integrable system in $S^4$ with its induced folded symplectic structure. Note that this is an example of an integrable system on a singular symplectic manifold which is not $b$-symplectic, as shown by the obstructions in \cite{GMP1} and \cite{MO}.

\subsection{Symplectic manifolds with fibrating boundary}

Consider a symplectic manifold with boundary such that close to the boundary the symplectic form tends to degenerate and admits adapted Martinet-Darboux charts such that the boundary has local equation  $x_1=0$ and the symplectic form degenerates on the boundary with the following local normal form:   $$\omega= x_1 dx_1 \wedge dy_1 + dx_2 \wedge dy_2 + \ldots + dx_n \wedge dy_n. $$
Let us take as starting point some integrable system naturally defined on a manifold with boundary. Assume the folding hypersurface fibrates by circles over a compact symplectic base (origami type). It would be enough to consider an integrable system on the $(2n-2)$-symplectic base $f_2, \ldots f_n$ and add $t^2$ as $f_1$.
The set $(t^2, f_2, \ldots, f_n)$ defines a folded integrable system. Observe that complete integrability comes as a consequence of Theorem \ref{thm:moser}.

\subsection{Product of folded surfaces with symplectic manifolds endowed with integrable systems}
Take an orientable surface $\Sigma$, and $\omega$ a non-vanishing two form. Denote $t$ any function in $\Sigma$ which is transverse to the zero section. The critical set is a finite number of closed curves $\gamma_j, \enspace j=1,...,k$. Then the function $t^2$ defines a folded integrable system in $(\Sigma, t\omega)$, where $t\omega$ is a folded symplectic structure. Let $F=(f_1,...,f_n)$ be an integrable system in a symplectic manifold $(M^{2n},\omega_1)$. Then $(t^2,f_1,...,f_n)$ defines a folded integrable system in the manifold $M^{2n}\times \Sigma$ endowed with the folded symplectic form $\omega_f=t\omega + \omega_1$. In fact, taking any $(n+1)$-tuple of the form $(t^2\sum_{i=1}^{n} \lambda_i f_i, f_1,...,f_n)$ for some non trivial $n$-tuple of constants $\lambda_i$  yields a folded integrable system. The critical set is of the form $Z= \sqcup_{j=1}^k \gamma_j \times M^{2n}$.

\subsection{Origami templates}
The study of toric folded symplectic manifolds was initiated in \cite{anarita} in the origami case (see \cite{H} for the general case).

Toric actions and integrable systems have always been hand-in-hand. In particular the action-angle coordinate theorem proof that we will provide in this article uses intensively
this correspondence. So in particular a toric manifold provides examples of integrable systems which are described by a global Hamiltonian action of a torus. Indeed any integrable system can be semilocally described in these terms (as we will see in the next section).

The classical theory of toric symplectic manifolds is closely related to a theorem by Delzant \cite{Delzant} which gives a one-to-one correspondence between toric symplectic manifolds and a special type of convex polytopes (called Delzant polytopes) up to equivalence. Grosso modo, toric symplectic manifolds can be classified by their moment polytope, and their topology can be read directly from the polytope in terms of equivariant cohomology. In \cite{holm, holm2} the authors examine the toric origami case and describe how toric origami manifolds can also be classified by their combinatorial moment data.

Origami templates form a visual way to describe toric origami manifolds and thus in particular integrable systems on folded symplectic manifolds.

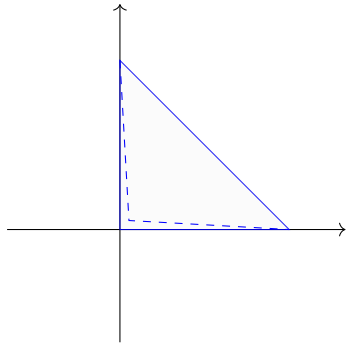
\begin{figure}[!h]
\begin{center}
\begin{tikzpicture}[scale=3]

\coordinate  (A) at (0,0);
\coordinate (B) at (3/4,0);
\coordinate (C) at (0,3/4);

\draw[->] (-1/2,0) -- (1,0) node[right] {};
\draw[->] (0,-1/2) -- (0,1) node[above] {};

\draw[-, fill=black!15, opacity=0.1] (A)--(B)--(C)--cycle;
\draw[-, opacity=.8, color=blue] (A)--(B)--(C)--cycle;

\draw[dashed,color=blue] (C)--(0.04,0.04)--(B);

\end{tikzpicture}
\end{center}
\caption{Origami template of Example 4.2}
\end{figure}

 Toric origami manifolds are classified by combinatorial origami templates which overlap Delzant's polytopes in an special way providing pictorially beautiful examples of folded integrable systems.

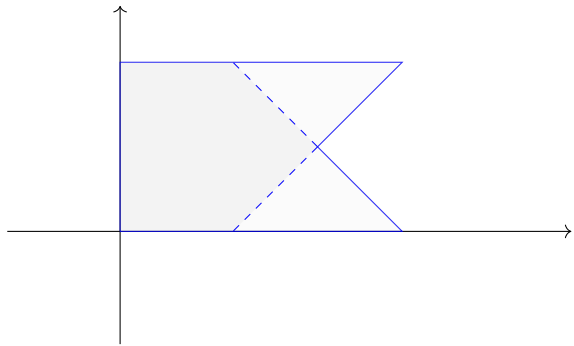
\begin{figure}[!h]
\begin{center}
\begin{tikzpicture}[scale=3]

\coordinate  (A) at (0,0);
\coordinate (B) at (3/4,0);
\coordinate (C) at (0,3/4);

\draw[->] (-1/2,0) -- (2,0) node[right] {};
\draw[->] (0,-1/2) -- (0,1) node[above] {};

\draw[-, fill=black!15, opacity=0.1] (A)--(0,3/4)--(1.25,3/4)--(7/8,3/8)--(1.25,0)--(A);
\draw[-, opacity=.8, color=blue] (A)--(0,3/4)--(1.25,3/4)--(7/8,3/8)--(1.25,0)--(A);

\draw[dashed,color=blue] (7/8,3/8)--(4/8,0);
\draw[dashed,color=blue] (7/8,3/8)--(4/8,3/4);

\fill[fill=black!30, opacity=0.1] (A)--(0,3/4)--(4/8,3/4)--(7/8,3/8)--(4/8,0)--(A);

\end{tikzpicture}
\end{center}
\caption{Origami template corresponding to the radial blow-up of two Hirzebruch surfaces.}
\end{figure}

\subsection{The folded spherical pendulum}

Consider the spherical pendulum on $S^2$ defined as follows: Take spherical coordinates $(\theta,\phi)$ with $\theta \in (0,\pi)$ and $\phi \in (0,2\pi)$ if we denote each momentum  as $P_\theta$ and $P_\phi$ respectively,  the Hamiltonian function is
$$ H= \frac{1}{2} (P_\theta^2+\frac{1}{\sin^2\theta} P_\phi^2)+\cos \theta. $$
Instead of taking the standard symplectic form in $T^*S^2$ we consider the folded symplectic form
$$ \omega= P_\phi dP_\phi \wedge d\phi + dP_\theta \wedge d\theta. $$
Computing the Hamiltonian vector field associated to $H$ we get
$$ X_H= \frac{1}{\sin^2 \theta} \pp{}{\phi} + P_\theta \pp{}{\theta} +(\sin \theta + \frac{\cos \theta}{\sin^3\theta} P_\phi^2) \pp{}{P_\theta}. $$
This vector field clearly commutes with $\pp{}{\phi}$, which is the Hamiltonian vector field of $f=P_\phi^2$. Observe furthermore that
$$ dH \wedge dP_\phi^2= -(\sin \theta +\frac{\cos \theta}{\sin^3 \theta}P_\phi^2)2P_\phi d\theta \wedge dP_\phi, $$
which is nondegenerate on a dense set of $M$ and on a dense set of $Z$ when seen as a section of the second exterior product of  the folded cotangent bundle.
 The manifold is in fact $M=T^*(S^2\setminus \{N,S\})$, i.e. we are taking out the poles of the sphere. In this sense $M$ is equipped with an origami symplectic form: the critical set is $T^*(S^2\setminus \{N,S\})$ and the null line bundle is an $S^1$ fibration generated by $\pp{}{\varphi}$.

Observe that dynamically this system is different from the standard spherical pendulum. When $P_\phi=0$, the vector field can have a non vanishing $\pp{}{\phi}$ component.

\subsection{A folded integrable system which cannot be modelled as a $b$-integrable system}

Consider $S^2$ with the folded symplectic form $\omega=hdh\wedge d\theta$. A folded function whose exterior derivative is a non-vanishing  one-form (when considered as section of the folded cotangent bundle) on a dense set of $M$ and of $Z$ defines a folded integrable system.
Take for instance $f= \cos \theta h^2$, which satisfies this condition. Computing its Hamiltonian vector field we obtain
$$X_f= h\sin \theta \pp{}{h} + 2\cos \theta \pp{}{\theta}. $$
This vector field vanishes at some points in the critical locus $Z=\{h=0\}$. A $b$-integrable system on a surface $\Sigma$ is defined by a function $f=c\log{h}+g$ with $g\in C^\infty (\Sigma)$. In particular its Hamiltonian vector field cannot vanish at any point on the critical hypersurface, as it is happening in this example of folded integrable system. {Thus, even if the structure $hdh\wedge d\theta$ can be seen as the desingularization of $\frac{1}{h}dh\wedge d\theta$, the dynamics of this folded integrable system cannot be modeled using the $b$-symplectic structure.}

\subsection{Cotangent lifts for folded symplectic manifolds}\label{fcotangent}

In this section we describe the cotangent lift in the set-up of folded symplectic manifolds.

When the group acting on the base is a torus this  procedure provides examples of folded integrable systems.

Consider a Lie group $G$ acting on $M$ by an action $\phi: G\times M \longrightarrow M$.
\begin{defi}
The cotangent lift of $\phi$ is the action on $T^*M$ given by $\hat \phi_g := \phi_{g^{-1}}^*$, where $g \in G$.
\end{defi}
The following commuting diagram holds:
\begin{align*}
\begin{diagram}
\node{T^\ast M}  \arrow{e,t}{\hat{\phi_g}}
\arrow{s,l}{\pi}
\node{T^\ast M} \arrow{s,r}{\pi}\\
\node{ M}  \arrow{e,t}{\phi_g}  \node{M}
\end{diagram}
\end{align*}
where $\pi$ is the projection from $T^*M$ to $M$. The cotangent bundle has the symplectic form $\omega=-d\lambda$ where $\lambda$ is the Liouville form. This form is defined by the property $ \langle \lambda_p , v \rangle = \langle p, (\pi_p)_*(v)\rangle$, where $v \in T(T^*M)$ and $p \in T^*M$. It can be shown easily that the cotangent lift is a Hamiltonian action with momentum map $\mu:T^*M \rightarrow \mathfrak{g}^*$ given by
$$  \langle\mu(p),X \rangle := \langle \lambda_p ,X^\#|_{p} \rangle =\langle p,X^\#|_{\pi(p)}\rangle.   $$
Here $X^\#$ denotes the fundamental vector field of $X$ associated to the action. The Liouville form is invariant by the action which implies the invariance of the moment map. In particular, the map is Poisson.

The construction called $b$-symplectic cotangent lift for $b$-symplectic manifolds done in \cite{KM} can be similarly done in the folded symplectic case which we will do below.

For the standard Liouville form in the folded cotangent bundle, the singularity is in the base space, and we would like to have it on the fiber. A different form, that we call \textbf{twisted folded Liouville form} can be defined on $T_V^*S^1$ with coordinates $(\theta,p)$:
$$  \lambda_{tw}= \frac{p^2}{2} d\theta_1.   $$
this way the singularity is in the fiber, and we can apply it to define a folded cotangent lift on the torus. Let $\T^n$ be the manifold and the group acting by translations, and take the coordinates $(\theta_1,...,\theta_n, a_1,...,a_n)$ on $T^*M$. The standard symplectic Liouville form in these coordinates is
$$ \lambda = \sum_{i=1}^n p_i d \theta_i. $$

The moment map $\mu_{can}: T^\ast \T^n \to \mathfrak{t^\ast} $ of the lifted action with respect to the canonical symplectic form is
$$ \mu_{can}(\theta, p) = \sum_i p_i d\theta_i $$
where the $\theta_i$ are seen as elements of $ \mathfrak{t^\ast}$. In fact, one can identify the moment map as just the projection of $T^*\T^n$ into the second component since $T^*\T^n \cong \T^n \times \mathbb{R}^n$.

\hfill \newline
This torus action in the cotangent bundle of the torus can be seen as a folded-Hamiltonian action with respect to a folded symplectic form. Similarly to the Liouville one-form we define the following singular form away from the hypersurface $Z=\{p_1 = 0\}$~:
$$ p_1^2 d \theta_1 + \sum_{i=2}^n p_i d\theta_i.$$
The negative differential gives rise to a folded symplectic form called twisted folded symplectic form on $T^\ast \T^n$:
\begin{equation*}\label{twistedform}
 \omega_{tw, f}:=p_1 d\theta_1\wedge d p_1 + \sum_{i=2}^n d\theta_i\wedge dp_i  .
\end{equation*}
The moment map is then
\begin{equation*}\label{eq:bmucan}
\mu_{tw, f}=(p_1^2,p_2, \ldots, p_n),
\end{equation*}
where we identify $\mathfrak{t^\ast}$ with $\R^n$ as before.

We call this lift the \textbf{folded cotangent lift}. Note that, in analogy to the symplectic case, the components of the moment map define a folded integrable system on $(T^\ast \T^n, \omega_{tw, f})$.
\hfill \newline

\begin{rem}
As we will see  in the next section, the folded cotangent lift does not  always serve as the semilocal model for an integrable system in the neighbourhood
of a Liouville torus, in contrast to what happens in symplectic and $b$-symplectic geometry.
\end{rem}

\section{Construction of action-angle coordinates}

In this section we prove existence of action-angle coordinates for singular symplectic manifolds of order one. One may have the temptation to use the desingularization and the action-angle coordinate theorem proved in \cite{KMS} to conclude. However, as we saw in previous sections, not every folded integrable system can be seen as a desingularized $b$-integrable system and  thus a complete proof is needed.

\subsection{Topology of the integrable system}
We first show that for a folded integrable system there is a foliation by Liouville tori in the neighborhood of a regular fiber of the integrable system. For this it is important to observe the following:

 Arguing as in the proof of Lemma \ref{normalint}, the kernel of $i^*\omega$ is generated by the joint distribution of the Hamiltonian vector fields of $f_1,...,f_n$ at every point $p$ of a neighborhood of a regular fiber. Hence, we can assume that $f_1=t^2/2$ for some semi-local coordinate $t$ defining $Z$.  The foliation given by the Hamiltonian vector fields of $F$ coincides with the foliation described  by the level sets of $\bar{F}=(t,f_2,...,f_n)$ because by definition the Hamiltonian vector field of $t^2/2$ is tangent to the level sets of this $t^2/2$, and hence also to the level sets of $t$. The same argument in \cite{LMV} gives a commutative diagram

\begin{center}
\begin{tikzcd}
U \arrow{r}{\varphi} \arrow[swap]{dr}{\bar F} & \T^n\times B^n  \arrow{d}{\pi} \\%
 & B^n
\end{tikzcd}
\end{center}
and proves:

\begin{prop}
 Let $p \in Z$ be a regular point of a folded-integrable system $(M,\omega,F)$. Assume that the integral manifold $\mathcal{F}_p$ is compact. Then there is neighborhood $U$ of $\mathcal{F}_p$ and a diffeomorphism
 $$ \varphi: U\cong \mathbb{T}^n\times B^n $$
 which takes the foliation $\mathcal{F}$ to the trivial foliation $\{ \mathbb{T}^n \times \{b \} \}_{b\in B^n }$.
\end{prop}

\begin{figure}
\centering
\begin{tikzpicture}

\pgfmathsetmacro{\sizer}{1.1}
\pgfmathsetmacro{\basept}{5}	

\pgfmathsetmacro{\xone}{0}
\pgfmathsetmacro{\xtwo}{1.5}
\pgfmathsetmacro{\xthree}{3}
\pgfmathsetmacro{\xfour}{4.5}
\pgfmathsetmacro{\xfive}{6}
\pgfmathsetmacro{\xsix}{7.5}
\pgfmathsetmacro{\xseven}{9}

\pgfmathsetmacro{\ymid}{4.25}
\pgfmathsetmacro{\ytop}{5.65}
\pgfmathsetmacro{\ybottom}{2.85}

\DrawFilledDonutops{\xone}{\ymid}{.55*\sizer}{1.2*\sizer}{-90}{vlred}{very thick}{white}
\DrawFilledDonutops{\xtwo}{\ymid}{.55*\sizer}{1.2*\sizer}{-90}{vlred}{very thick}{white}
\DrawFilledDonutops{\xthree}{\ymid}{.55*\sizer}{1.2*\sizer}{-90}{vlred}{very thick}{white}
\DrawFilledDonutops{\xfour}{\ymid}{.55*\sizer}{1.2*\sizer}{-90}{vlblue}{very thick}{white}
\DrawFilledDonutops{\xfive}{\ymid}{.55*\sizer}{1.2*\sizer}{-90}{vlred}{very thick}{white}
\DrawFilledDonutops{\xsix}{\ymid}{.55*\sizer}{1.2*\sizer}{-90}{vlred}{very thick}{white}
\DrawFilledDonutops{\xseven}{\ymid}{.55*\sizer}{1.2*\sizer}{-90}{vlred}{very thick}{white}

\DrawDonut{\xone}{\ymid}{.55*\sizer}{1.2*\sizer}{-90}{dred}{very thick}
\DrawDonut{\xtwo}{\ymid}{.55*\sizer}{1.2*\sizer}{-90}{dred}{very thick}
\DrawDonut{\xthree}{\ymid}{.55*\sizer}{1.2*\sizer}{-90}{dred}{very thick}
\DrawDonut{\xfour}{\ymid}{.55*\sizer}{1.2*\sizer}{-90}{dblue}{very thick}
\DrawDonut{\xfive}{\ymid}{.55*\sizer}{1.2*\sizer}{-90}{dred}{very thick}
\DrawDonut{\xsix}{\ymid}{.55*\sizer}{1.2*\sizer}{-90}{dred}{very thick}
\DrawDonut{\xseven}{\ymid}{.55*\sizer}{1.2*\sizer}{-90}{dred}{very thick}

\draw[very thick,magenta,vdred, ->](\xone, \ybottom - 0.3) -- +(0, -1);
\draw[very thick, vdred,  ->] (\xtwo, \ybottom - 0.3) -- +(0, -1);
\draw[very thick, vdred, ->] (\xthree, \ybottom - 0.3) -- +(0, -1);
\draw[very thick, blue, ->] (\xfour, \ybottom - 0.3) -- +(0, -1);
\draw[very thick,vdred, ->] (\xfive, \ybottom - 0.3) -- +(0, -1);
\draw[very thick,vdred, ->] (\xsix, \ybottom - 0.3) -- +(0, -1);
\draw[very thick,vdred, ->] (\xseven, \ybottom - 0.3) -- +(0, -1);
\draw[ultra thick, red](\xone - 1, \ybottom - 1.5) -- (\xseven + 1, \ybottom - 1.5);

\end{tikzpicture}
\caption{Fibration by Liouville tori: In the middle fiber (in blue) of the point $F(p)$, the neighbouring Liouville tori in red.}
\end{figure}
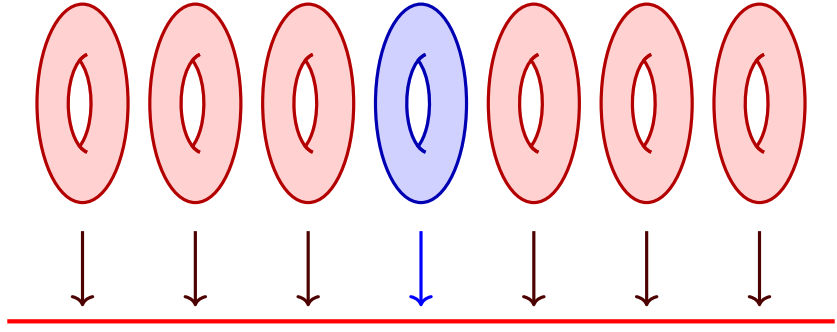

 We now  prove a Darboux-Carath\'eodory theorem for folded symplectic manifolds to (locally) complete a set of folded functions which commute with respect to $\omega$.
We do this applying the arguments of the proof of Darboux theorem provided in \cite{Arn}. The Darboux-Carathéodory theorem will be a key point in the proof of existence of action-angle coordinates.

\begin{theorem}[Folded Darboux-Carath\'eodory theorem]\label{thm:Cara}
    Let $p\in Z$ be a point of the folding hypersurface of a folded symplectic manifold $(M,\omega)$ and let $t$ be the function defining $Z$. Consider $f_1,...,f_n$ to be $n$ folded functions whose Hamiltonian vector fields are smooth, independent at $p$ and commute pairwise with respect to $\omega$. Then in a neighborhood $U$ of $p$ there exists $n$ functions $q_1,...,q_n$ such that near $p$ the folded symplectic form is written as
\begin{equation}\label{eq:foldDC}
    \omega= \sum_{i=1}^n dq_i \wedge df_i.
\end{equation}
 A system of coordinates is given by $q_1,...,q_n$ and some coordinates $t,y_2,..,y_n$ such that the $f_i$ only depend on the latter.
\end{theorem}

\begin{rem}
 Note that the functions $f_1,...,f_n$ do not define a set of $n$ independent coordinate functions in $U$ and, thus, Equation \eqref{eq:foldDC} does not correspond to a symplectic form but a folded symplectic form.
\end{rem}

\begin{proof}
In order to construct this decomposition we first construct the folded symplectic conjugate of the function $f_1$ following the classical recipe which we can find in \cite{Arn}.

In $U$ the foliation $\mathcal{F}$ induced by the level sets of $(t,f_2,...,f_n)$ coincides with the one generated by $\mathcal{D}=\langle X_1,...,X_n \rangle$, where $X_i$ denotes the Hamiltonian vector field of $f_i$. Take $B$ a submanifold of dimension $n$ containing $p$ and transverse to this foliation. We will now construct a function $q_1$ such that $X_1=\pp{}{q_1}$ and $dq_1(X_j)=\delta_{1j}$.

For each point $m\in B$, there is a leaf $L_m \subset U$ of the foliation $\mathcal{F}$ containing $p$. Each $L_m$ is foliated by $n-1$ dimensional leaves, induced by the foliation integrating $\mathcal{D}'=\langle X_2,...,X_n\rangle$. Denote by $L_m' \subset L_m$ the $n-1$ dimensional leaf containing $M$. For some $\varepsilon$, the flow $\phi_1^t$ of $X_1$ is defined for a smaller neighborhood $U'$ for $|t|<\varepsilon$ and with its image contained in $U$. For a point $x\in L_m\cap U'$, there is some $x' \in L_m'\cap U'$ and a $t'$ such that $|t'|<\varepsilon$ and $\phi_1^{t'}(x')=x$. Define the function:
\begin{align*}
q_1:U' &\longrightarrow \mathbb{R}\\
	x &\longmapsto t'(x).
\end{align*}
We claim that $dq_1(X_j)=\delta_{ij}$. First, observe that the flow $\phi_1^t$ preserves the foliations induced by $D$ and $D'$ because of the commuting conditions given by the integrable system. This implies that $q_1$ is constant along such foliation and hence $dq_1(X_j)=0$ if $j\neq 1$. The definition of $t'$ yields $dq_1(X_1)=1$.

  We can now fix both $q_1$ and $f_1$, and apply iteratively this construction for the flow of the Hamiltonian vector field of each of the remaining functions $f_2,...,f_n$. We get smooth functions $q_2,...,q_n$ such that $dq_i(X_j)=\delta_{ij}$. At the points in $U\setminus Z$, the functions $f_1,...,f_n,q_1,...,q_n$ do form a set of coordinates. This implies that in $U\backslash Z$, where $\omega$ is symplectic, we have
$$\omega=\sum_{i=1}^n dq_i\wedge df_i.$$
But both $\omega$ and the functions $q_i,f_i$ are smooth and defined along  $U$, hence this expression extends to $U$. Extend the functions $q_i$ to a set of coordinates $(q_1,...,q_n,y_1,y_2,...,y_n)$. We can assume that $y_1=t$ is a defining function of $Z$, since the $q_i$ are coordinates along the level sets of $F$: the vector fields $\pp{}{q_i}$ are the Hamiltonian vector fields of the $f_i$. The fact that $\omega(X_i,X_j)=0$ implies that $df_i(X_j)=\pp{}{q_j}(f_i)=0$. This proves that the $f_i$ depend only on $(t,y_2,..,y_n)$.
\end{proof}
In contrast with the Darboux-Carathéodory in symplectic and $b$-symplectic geometry, one can not obtain a canonical normal form as in Martinet's theorem. This is a consequence of the fact that when you fix several commuting folded functions, various of those functions can have a Hamiltonian vector field with a non-vanishing component in the null line bundle.
\begin{rem} The classical Darboux-Carathéodory theorem considers a set of $k<n$ commuting independent functions  $f_1,...,f_k$. The same proof can be adapted in this situation and  the same theorem applies for a set of $k<n$ commuting functions which are independent $f_1,...,f_k$. We can find then a set of coordinates $(q_1,...,q_k,y_1,...,y_k,x_{k+1},y_{k+1},...,x_n,y_n)$  such that {$\omega=  \sum_{i=1}^k dq_i \wedge df_i + \sum_{i=k+1}^n dx_i \wedge dy_i$.} This form of Darboux-Carathéodory theorem is convenient for the study of non-commutative integrable systems (see for instance \cite{KM2}).
\end{rem}

\subsection{Equivariant relative Poincaré's lemma for folded symplectic forms}

We start this section with some lemmas which we will need for the proof of the action-angle theorem. They concern Relative Poincaré's lemma for folded symplectic forms and their
equivariant versions.
Recall  from \cite[page 25]{W}.
\begin{theorem}[Relative Poincar\'e lemma]
Let $N \subset M$ a closed submanifold and $i:N \hookrightarrow M$ the inclusion map. Let $\omega$ a closed $k$-form on $M$ such that $i^*\omega=0$. Then there is a $(k-1)$-form $\alpha$ on a neighborhood of $N$ in $M$ such that $\omega= d\alpha$.
\end{theorem}

This Relative Poincaré lemma can be used in the particular case in which the form is folded and the submanifold is a Liouville torus.

\begin{prop}\label{prop:relative}
    In a neighborhood $U(L)$ of a Liouville torus  the folded symplectic form can be written
    $$ \omega= d\alpha. $$

    If $\omega$ is invariant by a compact group action, $\alpha$ can be assumed to be invariant by the same compact group action.
\end{prop}

\begin{proof}

Let  $i:L \hookrightarrow M$ be the natural inclusion of the Liouville torus on the folded symplectic manifold, since $i^*\omega=0$ we may apply the following relative Poincaré theorem.

Let us check  that the hypotheses of the Relative Poincaré lemma are met. The form is closed and we only need to check $i^*\omega=0$. Since every $Y_i$ is Hamiltonian with Hamiltonian function $\sigma_i$, we obtain that $\iota_{Y_i} \omega=d \sigma_i $.  And therefore the tangent space to $L$ is generated by $Y_1,...,Y_n$. However, we know that $i^*d\sigma_i=0$, since $L$ is the level set of the integrable system. This implies that $\iota_{Y_i}i^*\omega=0$ and hence $i^*\omega=0$.

Now  define the averaged $\bar \alpha $ as
    $$ \bar\alpha=\int_{G} \rho_g^*\alpha  d\mu, $$ where $\mu$ is a Haar measure and $\rho_g$ is the group action.
This 1-form is $G$-invariant, { and  as $\rho$ preserves $\omega$, we can write,}

    \begin{align*}
        \omega&= \int_{G} \rho_g^*\omega d\mu \\
              &= \int_{G} d\rho_g^*\alpha d\mu
          \end{align*}
Thus $ \omega  = d (\int_{G} \rho_g^*(\alpha) d\mu)$.
In particular this proves that the primitive $\bar\alpha$
 is invariant by the action. i.e,     for any $Y_i$ fundamental vector field of the torus action one obtains,
    $ \mathcal{L}_{Y_i}\bar\alpha=0. $ Thus finishing the proof of the proposition.
    \end{proof}

\subsection{Statement and proof of the action-angle coordinate theorem}
We proceed now with the statement and the proof of the action-angle theorem.
\begin{theorem}\label{AA}
Let $F=(f_1,...,f_n)$ be a folded integrable system on $(M,\omega)$ and $p\in Z$ a regular point in the folding hypersurface. We assume the integral manifold $\mathcal{F}_p$ containing $p$ is compact. Then there exist an open neighborhood $U$ of the torus $\mathcal{F}_p$ and a diffeomorphism
$$ (\theta_1,...,\theta_n, t, b_2,...,b_n):U\rightarrow \T^n \times B^n, $$
where $t$ is a defining function of $Z$ and such that
$$ \omega_U= \sum_{i=1}^n d\theta_i \wedge dp_i. $$
where the $p_i$ are folded functions which depend only on $(t,b_2,...,b_n)$ (and so do the $f_i$).

The $S^1$-valued functions
$$ \theta_1,...,\theta_n $$
are called angle coordinates and the $\R$-valued folded functions
$$p_1,p_2,...,p_n $$
are called { folded action coordinates}.
\end{theorem}

\begin{rem}
Comparing this theorem with the analog in  \cite{KMS} observe that unlike the $b$-symplectic case, the expression of $\omega$ in a neighborhood of the Liouville torus is not in a folded Darboux-type form.
\end{rem}
Besides the lemmas in the former subsection we will need the following technical lemma.

In \cite{LMV} (see Claim 2 in page 1856) it is shown that given a complete vector field $Y$ of period $1$ and a bivector field $P$ such that $\mathcal{L}_{Y}\mathcal{L}_{Y}P=0$ then $\mathcal{L}_{Y}P=0$. If instead of a bivector field we take a $2$-form, the proof can be easily adapted as follows.
\begin{lemma}\label{lmv}
	If $Y$ is a complete vector field of period $1$ and $\omega$ is a $2$-form such that $\mathcal{L}_{Y}\mathcal{L}_{Y}\omega=0$ then $\mathcal{L}_{Y}\omega=0$.
\end{lemma}
\begin{proof}
	Denote $v=\mathcal{L}_{Y}\omega$. Denote $\phi_t$ the flow of $Y$. For any point $p$ we can write
	\begin{align*}
		\frac{d}{dt}\left(  \phi_t^*\omega_{\phi_{-t}(p)} \right)&= {(\phi_t)}^* ( \mathcal{L}_Y{\omega_{{\phi}_{-t}(p)})} \\
		&= {\phi_t}^* v_{{\phi}_{-t}(p)} \\
		&= v_p
	\end{align*}
	In the last equality we used that $\mathcal{L}_Yv=0$. Integrating we obtain
	$$ (\phi_t)^*\omega_{\phi_{-t}(p)}= \omega_p + tv_p. $$
	At time $t=1$ the flow is the identity because $Y$ has period 1 and hence $v_p=0$.
\end{proof}

We now proceed to the action-angle theorem proof.
\begin{proof}
 The vector fields $X_{f_1},...,X_{f_n}$ define a torus action on each Liouville torus $\mathbb{T}^n \times \{b\}_{b\in B^n}$. We would like an action defined in a neighborhood of the type $\T^n \times B^n$. For the first part of the proof we follow the proofs in \cite{LMV} and \cite{KMS} and construct a toric action. For this we consider the classical action of the joint-flow (which is an $\mathbb{R}^n$-action) and prove uniformization of periods to induce a $\mathbb{T^n}$-action.

We denote by $\varphi_{i}^{t}$ the time-$t$-flow of the Hamiltonian vector fields $X_{f_i}$. Consider the joint flow of these Hamiltonian vector fields.
\begin{align*}
\varphi :\mathbb{R}^n \times ( \mathbb{T}^n \times B^n ) &\longrightarrow \mathbb{T}^n \times B^n \\
 \big((t_1,\ldots,t_n),(x, y)\big) &\longmapsto \varphi_1^{t_1}\circ\dots\circ \varphi_n^{t_n}(x, y).
\end{align*}

The vector fields $X_{f_i}$ are complete and commute with one another so this defines an $\R^n$-action on $\T^n \times B^n$. When restricted to a single orbit $\T^n \times \{b\}$ for some $b \in B^n$, the kernel of this action is a discrete subgroup of $\R^n$, a lattice $\Lambda_{b}$. We call  $\Lambda_{b}$ the period lattice of the orbit. The rank of $\Lambda_b$ is $n$ because the orbit is assumed to be compact. \\

The lattice $\Lambda_b$ will in general depend on $b$. The idea of uniformization of periods is to modify the action to get constant isotropy groups such that $\Lambda_b = \mathbb{Z}^n$ for all $b$. For any $b \in B^{n-1}\times\{0\}$ and any $a_i \in \mathbb{R}$ the vector field $\sum a_i X_{f_i}$ on $\mathbb{T}_n \times \{b\}$ is the Hamiltonian vector field of the function
$$\sum_{i=1}^{n} a_i f_i.$$

To perform the uniformization we pick smooth functions
$$(\lambda_1, \lambda_2, \dots, \lambda_n): B^n \rightarrow \mathbb{R}^n$$
such that { $(\lambda_1(b), \lambda_2(b), \dots, \lambda_n(b))$ is a basis for the period lattice $\Lambda_b$ for all $b \in B^n$.}
 Such functions $\lambda_i$ exist such that they satisfy this condition (perhaps after shrinking $B^n$) by the implicit function theorem, using the fact that the Jacobian of the equation $\Phi(\lambda, m) = m$ is regular with respect to the $s$ variables.

We define a uniformized flow using the functions $\lambda_i$ as
\begin{align*}
\tilde \Phi :\R^n \times ( \T^n \times B^n ) &\to  \T^n \times B^n \\
 \big((s_1,\ldots,s_n),(x, b)\big) &\mapsto \Phi \left(\sum_{i=1}^n s_i \lambda_i(b), (x,b) \right).
\end{align*}
The period lattice of this $\R^n$ action is  $\mathbb{Z}^n$, and  therefore constant hence the initial action  clearly descends to  the quotient to define a new action of the group  $\mathbb{T}^n$.

We want to find now functions $\sigma_1,...,\sigma_n$ such that their Hamiltonian vector fields are precisely the ones constructed above $Y_i= \sum_{j=1}^n \lambda_i^j X_{f_j}$.
We compute the Lie derivative of the vector fields $Y_i$ using Cartan's formula:
\begin{align*}
    \mathcal{L}_{Y_i} \omega&= d\iota_{Y_i}\omega + \iota_{Y_i}d\omega \\
    &=d(- \sum_{j=1}^n \lambda_i^j df_j )  \\
    &= -\sum_{j=1}^n d\lambda_i^j \wedge df_j
\end{align*}
 We deduce that
$$\mathcal{L}_{Y_i}\mathcal{L}_{Y_i} \omega= \mathcal{L}_{Y_i}(-\sum_{j=1}^n d\lambda_i^j \wedge df_j) .$$
In the last equality we have used the fact that $\lambda_i^j$ are constant on the level sets of $F$.
Lemma \ref{lmv} applied to the vector fields $Y_i$ yields $\mathcal{L}_{Y_i}\omega=0$ and the folded-symplectic structure is preserved.

The next step is to prove that the collection of $1$-forms $\iota_{Y_i}\omega$  are exact in the neighbourhood of a Liouville torus.  So the new action is indeed Hamiltonian.
We apply proposition \ref{prop:relative} in a neighbourhood of a Liouville torus and the symplectic form $\omega$ can be written as
$\omega=d\alpha$.
Now since $\mathcal{L}_{Y_i}\omega=0$, consider the toric action generated by the vector fields $Y_i$. Applying the equivariant version of Proposition \ref{prop:relative} with the group $G=\mathbb T^n$ the form $\omega$ is $G$-invariant and we can find a new $\bar{\alpha}$ which is at the same time a primitive for the folded symplectic structure $\omega$ and $\mathbb T^n$-invariant.

%

 Cartan's formula yields:
\begin{align*}
    \iota_{Y_i}\omega &= -\iota_{Y_i}d\bar\alpha \\
                &= -d\iota_{Y_i}\bar\alpha.
\end{align*}
Thus we deduce that the fundamental vector fields $Y_i$ are indeed Hamiltonian with Hamiltonian folded functions $\iota_{Y_i}\bar\alpha$. Denoting by $\sigma_1,...,\sigma_n$ these Hamiltonian functions, they are now the natural candidates for  ``action" coordinates. { Each of these functions defines a smooth Hamiltonian vector field, so by definition they are all folded functions.

We are under the hypotheses of Theorem \ref{thm:Cara} (Darboux-Carath\'eodory theorem), so we can find a coordinate system
$$ (t,y_2,...,y_n,q_1,...,q_n) $$
and some folded functions $\sigma_i$ such that
$$ { \omega= \sum_{i=1}^n d\sigma_i \wedge dq_i, }$$
and the $\sigma_i$ depend only on $(t,...,y_n)$.}
{ The functions $\sigma_i$ were defined using an equivariant form $\alpha_i$ which is defined in a neighborhood of the whole regular fiber. Hence the $\sigma_i$ extend to all $U'=\sigma^{-1}(\sigma(U))$.}.  For the sake of simplicity we denote these extensions using the same notation. The Hamiltonian vector fields of $\sigma_i$ have period one, so the functions $q_i$ can be viewed as angle variables $\theta_i$. { It remains to check that, in  the extended functions,  $\omega$  can be written in the desired \emph{Darboux-type} form}.

Observe that  $\omega(\pp{}{\sigma_i},\pp{}{\theta_i})=\delta_{ij}$ in $U'$ by the own definition of $\theta_i$. By abuse of notation we denote by $X_{\theta_i}$ the vector fields which solve the equations: $\iota_{ X_{\theta_i}}\omega=-d\theta_i$. { By construction, the equality $\omega(Y_i,Y_j)=0$ holds in $U'$. This follows from the fact that $\omega$ is symplectic away from $Z$, and since $[Y_i,Y_j]=0$ we get that $\omega(Y_i,Y_j)=0$ in $ U'\setminus Z$ and hence the equality extends to all $U'$.}

We know, by the Darboux-Carath\'eodory coordinates, that $\omega(\pp{}{\sigma_i}, \pp{}{\sigma_j})= \omega(X_{\theta_i},X_{\theta_j})=0$ in the neighborhood $U$ of the regular point. Applying the definition of exterior derivative, using that $\omega$ is closed and that the vector fields commute we obtain:
\begin{align*}
	d\omega (X_{\theta_i}, X_{\theta_j}, X_{\sigma_k}) &= X_{\theta_i}(\omega(X_{\theta_j},X_{\sigma_k})) - X_{\theta_j}(\omega(X_{\theta_i},X_{\sigma_k})) \\
	&+ X_{\sigma_k}(\omega(X_{\theta_i},X_{\theta_j})) \\
	&=0
\end{align*}
Using that $\omega(X_{\sigma_i},X_{\theta_j})=\delta_{ij}$ for all $i$ and $j$, we obtain
$$  X_{\sigma_k}(\omega(X_{\theta_i},X_{\theta_j}))=0.  $$
 In particular, by using the joint flow  of the vector fields $X_{\sigma_k}$ we prove that the relation { $\omega(\pp{}{\sigma_i}, \pp{}{\sigma_j})=\{\theta_i,\theta_j\}=0 $} holds in the whole neighborhood $U'$.
We conclude that $\omega$ has the desired form
$$ \omega =\sum_{i=1}^n d\sigma_i \wedge d\theta_i. $$
{ We consider the change $p_i:=-\sigma_i$ so that we can write  $\omega$ in the form}
$$ \omega= \sum_{i=1}^n d\theta_i \wedge dp_i. $$
Taking some coordinates $(t,b_2,...,b_n)$ in $B^n$, it is clear that the functions $(t,b_2,...,b_n, \theta_1,...,\theta_n)$ form a coordinate system and the $f_i$ only depend on $(t,b_2,..,b_n)$. This concludes the proof.
\end{proof}

\subsection{Desingularization and equivalence with cotangent models}\label{sec:yescotangent}

In \cite{KM} several examples of $b$-integrable systems are provided using the $b$-cotangent lift. In fact, the construction can be generalized to the context of $b^m$-symplectic manifolds.   From the definition of cotangent lift and the results in \cite{KM} which we recalled in subsection \ref{ssec:bcotangentlift}
we obtain:
\begin{prop}
The twisted $b$-cotangent lift of the action of an abelian group $G$ of rank $n$ on $M^{2n}$ yields a $b$-Hamiltonian action in $T^*M$. If the action is free or locally free, the twisted cotangent lift yields a $b$-integrable system.
\end{prop}

In \cite{convexitygmw} the desingularization of torus actions was explored in detail. As  a consequence of theorem 6.1 in \cite{convexitygmw} where an equivariant desingularization procedure is established for effective torus actions, we obtain the following \emph{desingularized models}.

\begin{prop}
The equivariant desingularization takes
the twisted $b$-cotangent lift of an action of a torus $\mathbb T^n$ on $M$
to a twisted folded cotangent lift model.

\end{prop}
\begin{rem}In \cite{arnaueva2} explicit desingularization formulae are given for action-angle coordinates of desingularized systems. These are  convenient for the refinement of KAM theory for singular symplectic manifolds.
\end{rem}
\begin{proof}
Denote $t$ the defining function of the critical hypersurface $Z$. The moment map of the action in the $S^1$ coordinate is a function of the form $f= c \log (\vert p\vert )$ for some constant $c$, where $p$ denotes the momentum coordinate in $T^*S^1$. Its Hamiltonian vector field is $X_f= c \pp{}{\theta}$.
Take $f'=\frac{cp^2}{2}$ as new momentum map component for the folded symplectic structure in $T^*M$.
\end{proof}

 This construction provides a machinery to produce examples of folded integrable systems via desingularization of $b$-integrable systems which are given by toric actions.
 However, we know that not all integrable system on a folded symplectic manifold comes from desingularization. Indeed a folded integrable system may not accept a folded cotangent model as we show in the next section.

\subsection{About equivalence with cotangent models: a case study}\label{sec:nocotangent}
For integrable systems in symplectic and $b$-symplectic geometry, the action-angle coordinates yields has a Darboux-type expression for the associated geometrical structure in a semi-local neighborhood. This does not always apply in the case of folded integrable systems, since we do not obtain an expression as in Martinet theorem:
$$ \omega=tdt\wedge dq + \sum_{i=2}^n dx_i\wedge dy_i $$
for some coordinates $(t,q,x_1,y_1,...,x_n,y_n)$. In general, the previous theorem cannot be simplified to obtain such an expression.

Indeed, assume that we can find  some action-angle type coordinates of the form, say, $\omega=tdt\wedge d\theta_1 + \sum_{i=2}^n dp_i \wedge d\theta_i$ for some $\mathbb{R}$-valued coordinates $p_i$ and $S^1$-valued coordinates $\theta_i$. In these coordinates, the null line line bundle $\ker i^*\omega|_Z$ is generated by $\langle \pp{}{\theta_1} \rangle$. This would imply that the null line bundle is a fibration by circles near the regular torus but we know this is not the case in general. We will now show an example of folded integrable system on a folded symplectic manifold whose null bundle has only two closed orbits.

\begin{ex}
 Take the mapping torus of $S^2$ by a smooth irrational rotation $\phi$, which is a symplectomorphism of $S^2$ equipped with the symplectic form $dh\wedge d\varphi$ which has only two periodic points. We get $S^2\times S^1$, and a cosymplectic manifold $(\alpha,\tilde \omega)$ where the form $\tilde \omega$ is obtained by gluing $dh\wedge d\varphi$ with $\phi^*(dh\wedge d\theta)$. It satisfies that $\ker \tilde \omega$ is a suspended vector field, which is an irrational rotation on each tori given by $h=c$ where $h$ is the height function on $S^2$.

By multiplying by $S^1$, we get $S^2\times S^1 \times S^1$, which can be endowed with the folded symplectic form
$$ \omega= { \sin \theta} d\theta \wedge \alpha + \tilde \omega. $$
The critical hypersurface is given by two copies of $S^2\times S^1$, at $\theta=0,\pi$, where $\ker i^*\omega|_Z=\ker \tilde \omega$.

The pair $(f_1,f_2)=({ \cos\theta}, h)$ defines a folded integrable system in $(S^2\times S^1\times S^1, \omega)$. Indeed, we have { $df_1\wedge df_2=-\sin\theta d\theta \wedge dh$}, which is non-vanishing in a dense set of $M$ and $Z$ as a section of $\Lambda^2(T_V^*M)$. The null line bundle of of $\omega$ is $\ker \tilde \omega|_Z$, which generates a vector field which has only two closed orbits at $h=1$ and $h=0$. We deduce that this folded integrable system does not admit a cotangent model. { In particular, in the normal form obtained in Theorem \ref{AA}, none of the functions $p_i$ is of the form $p_i=t^2$ for some defining function $t$ of $Z$.}
\end{ex}

This proves the following proposition.
\begin{prop}
Folded integrable systems do not admit, in general, cotangent models near a regular point.
\end{prop}
Typically, folded symplectic structures exhibited more flexibility (in the geometrical sense) than $b$-symplectic structures. This is captured by the fact that they adhere to an existence $h$-principle as proved Cannas \cite{anacannas} and in particular, all $4$-dimensional compact orientable manifolds admit a folded structure. On the other hand, the previous proposition can be seen as a rigidity phenomenon, which arises from considering dynamical aspects rather than geometrical ones. This rigidity arises from the existence of a canonical null foliation on the folding hypersurface. For $b$-symplectic manifolds, this null foliation is not canonical: it is defined up to Hamiltonian vector fields tangent to the leaves.  This explains why from this dynamical point of view, this flexibility allows to obtain canonical normal forms for $b$-integrable systems.

\section{ Constructions of integrable systems}

In this section, we study the existence of integrable systems on $b$-symplectic manifolds and their possible desingularization into folded integrable systems. We construct ad-hoc integrable systems on any $4$-dimensional $b$-symplectic manifold  whose critical locus admits a transverse Poisson $S^1$-action, starting from integrable systems defined on the leaf of a cosymplectic manifolds. In what follows  we will always assume that the symplectic foliation on the critical set $Z$ contains a compact leaf, and thus $Z$ is a symplectic mapping torus by \cite[Theorem 19]{GMP2}.  Furthermore, we will assume that the first singular integral induces an $S^1$-action in a neighborhood of $Z$ which is transverse to the symplectic foliation on $Z$. {In particular, the monodromy obtained by the first return map of the Hamiltonian vector field of the first integral induces a finite group action on the symplectic leaf of $Z$. The finite group action detects the points where the initial circle action is not free.}

\subsection{Structure of a $b$-integrable system in $Z$}

We start analyzing how a $b$-integrable system behaves on $Z$, the critical hypersurface of a $b$-symplectic manifold $(M,\omega)$.
\begin{claim}\label{leaf}
Let $F$ stand for a $b$-integrable system on a $b$-symplectic manifold $(M,\omega)$. Then for a fixed symplectic leaf $L$ of $Z$ there is a dense set of points in $L$ that are regular points of $F$.
\end{claim}

\begin{proof}
Assume that the set of regular points in a fixed leaf $L$ is not dense. Then we can find an open neighbourhood $U$ in $L$ which does not contain any regular point, i.e. $df_1 \wedge \dots \wedge df_n=0$ (when seen as a section of $\Lambda^n({}^bT^*M)$).  However,  in order for $F$ to define a $b$-integrable system,  one of the functions has to be a genuine (i.e., non-smooth) $b$-function in a neighborhood of $Z$. In other words,  $f=c\log |t|+g$ with $c\neq 0$ and $g$ a smooth function. We can assume that $f_1=f$ is a genuine $b$-function in a neighborhood $U'$ in $Z$ containing $U$. Since $c\neq 0$, it defines a Hamiltonian vector field whose flow is transverse to the symplectic leaf $L$. The function $f_1$ Poisson commutes with all the other integrals, and so the the flow of $f_1$ preserves $df_1 \wedge \dots \wedge df_n$.

Denote by $\varphi_t$ the flow of $X_{f_1}$. Then the set $V=\{\varphi_t(U')\enspace | \enspace t\in (0,\varepsilon)\}$ is an open subset of $Z$ where $df_1\wedge \dots \wedge df_n =0$. This is a contradiction with the fact that $F=(f_1,\dots,f_n)$ defines a $b$-integrable system.
\end{proof}

{ Once we take into the account that the first integral $f_1=c\log |t|$ induces an $S^1$-action, we can deduce the semi-local structure of the system.
\begin{prop}\label{s1action}
Let $(M,\omega)$ be a $b$-symplectic manifold admitting a $b$-integrable system such that $f_1=c\log |t|$ defines an $S^1$-action in the neighborhood of $Z$. Then $(f_2,...,f_n)$ induces an integrable system on each symplectic leaf $L$ on $Z$ which is invariant by the monodromy of the $S^1$-action.
\end{prop}
\begin{proof}
The fact that we may always assume that in a neighborhood of $Z$ the first integral is $f_1=c\log |t|$ follows from  remark 16 in \cite{KMS}(see also Proposition 3.5.3 in \cite{K}), where $c$ is the modular period of that connected component, and $f_2,...,f_n$ are smooth.  Observe that because $f_1$ is regular everywhere in a neighborhood of $Z$, the induced $S^1$-action has no fixed points.

{ By hypothesis, the Hamiltonian vector field $X_{c\log |t|}$ commutes with the Hamiltonian vector fields $X_{f_2},...,X_{f_n}$ which implies that the flow $\varphi_t$ of the $S^1$-action preserves each of the  functions $f_2,...,f_n$. The flow also preserves the symplectic foliation in $Z$. Thus, fixing a symplectic leaf $L$, the flow $\varphi_t$ satisfies $\varphi_t(L)\cong L$ and $\varphi_t^*(f_2,...,f_n)=(f_2,...,f_n)$. This shows that on each leaf the functions $f_2,...,f_n$ induce the same integrable system. In particular this integrable system in $L$ is preserved by the first return map of the monodromy in that fixed leaf, implying that the system is invariant by that finite group action.}
\end{proof}
%

\begin{rem}
{In the jargon of three-dimensional geometry, the connected components of the critical set $Z$ of a 4-dimensional manifold are Seifert manifolds with orientable base and vanishing Euler number. This follows from the fact that $Z$ is a mapping torus and that the first Hamiltonian vector field induces an $S^1$-action without fixed points.
}
\end{rem}}

\subsection{Construction of $b$-integrable systems}

Taking into the account the last remark, in order to construct $b$-integrable systems we will assume that $Z$ is the mapping torus of a periodic symplectomorphism of a compact leaf $L$ on $Z$. { This symplectomorphism defines a finite group action on $L$.  This is why in order to construct $b$-integrable systems on $4$ dimensional $b$-symplectic manifolds, we start by proving that we can always find a non-constant function  which is invariant under  a symplectic finite group action on a surface.}

\begin{claim}\label{inv}
Let $\mathbb{Z}_k$ be a finite group acting of a symplectic surface $\Sigma$. Then there exists a non-constant analytic function $F$ invariant by the group action.
\end{claim}

\begin{proof}
Take $f$ a generic analytic function in $\Sigma$. Consider the averaged function given by the \emph{averaging trick}:
$$ F(x):= \sum_{i=1}^{k-1} f(i.x) $$\label{averagingtrick}.
By construction this analytic function is invariant by the action of $\mathbb{Z}_k$. Given a point $p$ in $\Sigma$,  the differential of $F$ vanishes at $p$ if and only if $df_p+df_{2.p}+...df_{(k-1).p}=0$. Observe that for a generic $f$, there exists a point where this condition is not fulfilled. In particular, we deduce that $dF_p\neq 0$ at some point $p$, and hence $F$ is not a constant function.
\end{proof}

In the claim above we can replace the analytic condition by a Morse function $F$. See for instance \cite{Wass} for a proof of the existence and density of invariant Morse functions by the action of a compact Lie group.

\begin{theorem}\label{thm:bexist}
Let $(M,\omega)$ be a $b$-symplectic manifold of dimension $4$ with critical set $Z$ which is a mapping torus associated to a periodic symplectomorphism. Then $(M,\omega)$ admits a  $b$-integrable system.
\end{theorem}

\begin{proof}
In this case, a leaf of the critical set is a surface $L$. Take a neighborhood of $Z$ of the form $U=Z\times (-\varepsilon,\varepsilon)$. Denote by $X$ the Hamiltonian vector field of the function $\log t$ for some defining function of $Z$. By hypothesis, $X$ defines a Poisson $S^1$-action in $Z$ transverse to the leaves as studied in \cite{BKM}. This $S^1$-action can have some monodromy. Denote by $\alpha$ and $\beta$ the defining one and two forms of $\omega$ at $Z$. That is, in $U$ we can assume that $\omega$ has the form $\omega= \alpha \wedge \frac{dt}{t} + \beta$ with $\alpha \in \Omega^1(Z), \beta \in \Omega^2(Z)$. Recall that both forms are closed and $i^*_L \beta$ is a symplectic form in a leaf $L$ of $Z$.

 The critical set can be described as follows: There is an equivariant cover $L \times S^1 \times (-\varepsilon, \varepsilon)$ of $U$, and we denote by $p$ the projection to $U$. This equivariant cover can be equipped with the $b$-Poisson structure
$$ \omega= \pi_{Z_0}^*\tilde \alpha \wedge \frac{dt}{t} + \pi_{Z_0}^*\tilde \beta,   $$
where $\tilde \alpha= p^*\alpha$ and $\tilde \beta= p^*\beta$. Then $U$ is Poisson isomorphic by \cite[Corollary 17]{BKM} to the quotient of the equivariant by the action of a finite group $\mathbb{Z}_k$ in the leaf given by the return time flow of the $S^1$-action and extended trivially to the neighborhood $L \times S^1 \times (-\varepsilon,\varepsilon)$.

The action of $\mathbb{Z}_k$ acts by symplectomorphisms on $L$. By Lemma \ref{inv}, there is an analytic function $F$ in $L$ which is invariant by the action. In particular, $F$ can be extended to $\tilde F$ in all $Z$ by the $S^1$-action. If $\pi$ is the projection in $U=Z\times (-\varepsilon,\varepsilon)$ to the first component, then we extend $\tilde F$ to $U$ by considering $\pi^*\tilde F$ and denote it again $\tilde F$.

We construct in the neighborhood $U$ the pair of functions $(f_1,f_2)=(\varphi(t)c\log |t|, \varphi(t)\tilde F)$ in $U$. The function $\varphi(t)$ denotes a bump function which is constantly equal to $1$ for $t\in (-\delta, \delta)$ and constantly equal to $0$ for $|t|>\delta'$, with $\delta < \delta' < \varepsilon$. Observe the functions $f_1$ and $f_2$ are linearly independent in $^bT^*M$ in a dense set of $Z\times (-\delta',\delta')$. The Hamiltonian vector field of $\varphi(t)f_1$ generates the transverse $S^1$ action extended to $U$, and the Hamiltonian vector field of $\tilde F$ is tangent to the symplectic leaves in each $Z\times \{t_0\}$. Hence $\{f_1,\tilde F\}=0$. Now using the properties of the Poisson bracket we obtain
\begin{align*}
\{f_1,\varphi(t)\tilde F\}&= -\{\varphi(t)\tilde F, f_1\} \\
				&= \{ \varphi(t),f_1\} \tilde F + \{ \tilde F, f_1 \} \\
				&= 0 + 0,
\end{align*}
where we used that $f_1$ only depends on the coordinate $t$. We obtain an integrable system in the neighborhood of the critical locus $U$. To obtain an integrable system in all $M$, we do it as in the proof of existence of integrable systems in symplectic manifolds as shown by Brailov (cf. \cite{FT}). That is, cover $M\setminus U$ by Darboux balls, each of them equipped with a local integrable system of the form $f_i'=x_i^2+y_i^2$. By cutting off this system using a { function $\varphi (\sum_{i=1}^2( x_i^2+y_i^2) )$}, we can obtain for each Darboux ball a globally defined pair of functions $f_i=\varphi.f_i'$. We can now cover $M\setminus U$ by a finite amount of balls $B_i$ whose intersection is only the union of their boundaries. We choose $\varphi$ in each ball such that the locally defined integrable systems vanish in all derivatives exactly at these boundaries. The closed set of zero measure where the globally constructed $n$-tuple of functions are not independent is composed of the boundaries of the balls, and includes $Z \times \{-\varepsilon,\varepsilon\}$. { This is illustrated in Figure \ref{fig:balls}, where only some balls are depicted close to the boundary of $Z\times [-\varepsilon,\varepsilon]$. The closed set where the functions vanish are represented by  the black-colored boundaries.}

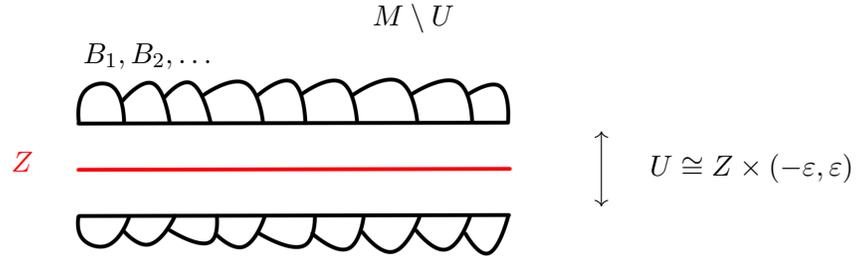
\begin{figure}[!h]
\begin{center}
\begin{tikzpicture}
     \node[anchor=south west,inner sep=0] at (0,0) {\includegraphics[scale=0.24]{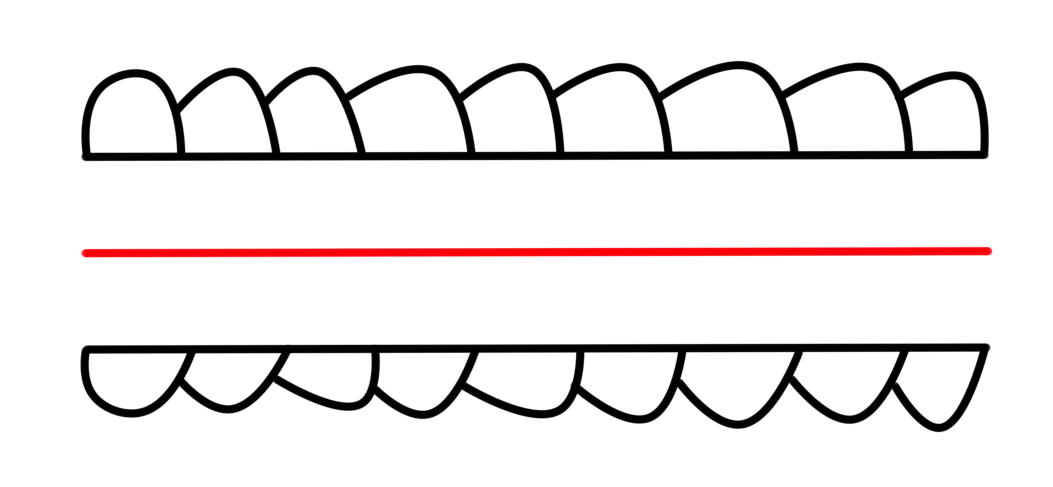}};
     \node at (9.5,1.5) {$U\cong Z\times (-\varepsilon,\varepsilon)$};
     \draw [<->] (7.5,1)--(7.5,2);
    \node at(5,3.5) {$M\setminus U$};
     \node at(1.5,3) {$B_1, B_2, \dots$};
     \node at (-3,0) {$\enspace$};
     \node[red] at (-0.2,1.6) {$Z$};
\end{tikzpicture}
\caption{Some Darboux balls filling $M\setminus U$.}
\label{fig:balls}
\end{center}
\end{figure}

This allows to glue the system in each ball and with the system we constructed in $U$, yielding a pair of commuting functions $F_1,F_2$ such that $dF_1\wedge dF_2\neq 0$ in a dense set of $M$ and $Z$.
\end{proof}

\begin{rem}\label{rem:highdimen}
The proof generalizes to higher dimensions as long as one can construct an integrable system in the symplectic leaf invariant by the finite group action. This is the content of Claim \ref{averagingtrick} for the case of a symplectic surface.
\end{rem}

\begin{rem}
{ The original construction of integrable systems in regular symplectic manifolds via the covering of Darboux balls yields an integrable system without any interesting property. However, the construction in Theorem \ref{thm:bexist} gives rise to a lot of examples of $b$-integrable systems that near the singular set $Z$ can be very rich from a semi-global point of view. }
\end{rem}

The following theorem is Theorem B in \cite{FMM}:
\begin{theorem}\label{thm:fmm}
Any cosymplectic manifold of dimension 3 is the singular
locus of orientable, closed, b-symplectic manifolds.
\end{theorem}
In particular, whenever the cosymplectic manifold has periodic monodromy, it can be realized as the critical locus of a $b$-symplectic manifold with a $b$-integrable system. { Theorem \ref{thm:fmm} requires specifically that $Z$ is connected. If we drop that requirement, there is a direct construction (Example 19 in \cite{GMP1}) to realize any cosymplectic manifold as the singular locus of a $b$-symplectic manifold that we will  introduce later.}

The proof of Theorem \ref{thm:bexist} can be adapted to obtain folded integrable systems in the desingularized folded symplectic manifold resulting from applying Theorem \ref{thm:desing}.

\begin{corollary}
Let $(M,\omega)$ be a $b$-symplectic manifold in the hypotheses of Theorem \ref{thm:bexist}. Then the desingularized folded symplectic manifold $(M,\omega_\varepsilon)$ admits a folded integrable system.
\end{corollary}
\begin{proof}
The desingularization given by Theorem \ref{thm:desing} sends $\omega$ to $\omega_\varepsilon$, which is a folded symplectic structure in $M$ with critical hypersurface $Z$. The induced structure on $Z$ remains unchanged: it is a cosymplectic manifold with compact leaves whose monodromy is periodic. The $S^1$-action generated by the modular vector field becomes the null line bundle of $\omega_\varepsilon$. Such line bundle is generated in a neighborhood of $Z$ by the Hamiltonian vector field of $t^2$, where $t$ is defining function of $Z$. By Claim \ref{averagingtrick}, there is an analytic function invariant by the first return map $X_{t^2}$. One can do exactly the same construction as in the proof of Theorem $\ref{thm:bexist}$, taking as first function $f_1=\varphi(t)t^2$ instead of $\varphi(t)\log |t|$ in the neighborhood $U$ of $Z$.
\end{proof}

\section{Global action-angle coordinates: Constructions and existence}

In this section we extend toric actions on  the symplectic leaf on the critical set of a $b$-symplectic manifold and folded-symplectic manifold to a toric action in the neighbourhood of the critical set $Z$. Thus obtaining global action-angle coordinates. For certain compact extensions of this neighbourhood we obtain global action-angle coordinates on the compact manifold. In doing so, we explore obstructions for the existence of global-action angle coordinates which lie on the critical set $Z$.

For a global toric action which we combine with the finite group transversal action given by the cosymplectic structure on $Z$ to produce an example of integrable system on any $b$-symplectic/folded symplectic manifold with toric symplectic leaves on the critical set.

By doing so, we explore the limitations that this construction has to admit an extension to a global toric action and thus admit global action-angle coordinates. This limitation lays on the topology of the critical set $Z$ which can be an obstruction for the global existence of action-angle coordinates.
In other words, this construction admits global action-angle coordinates if and only if the toric structure of the symplectic leaf of the critical set $Z$ extends to a toric action of the $b$-symplectic/folded symplectic manifolds.
Toric symplectic manifolds are well-understood thanks to \cite{GMPS} and \cite{log}.

In this section we will need to following lemma (which is Corollary 16 in \cite{GMPS}):

\begin{lemma}\label{lem:product} Let $(M^{2n}, Z, \omega)$ be a $b$-symplectic manifold with a toric action and $L$ a symplectic leaf of $Z$. Then $Z \cong L \times S^1.$
\end{lemma}

Let $L$ be a toric symplectic manifold of dimension $2n-2$ and let $F=(f_2,\dots,f_n)$ be its moment map.

We know from Delzant's theorem \cite{Delzant} that the image of $F$ is a Delzant polytope. From the definition of moment map the components of $F$ Poisson commute and are functionally independent so they form an integrable system on $L$.
Consider now $\phi$ be a symplectomorphism of $L$ which is equivariant with respect to the toric action and let $Z=L\times [0,1]/\sim$ be the cosymplectic manifold associated to it.
Extend the integrable system on $Z$ to an integrable  system on $Z$ just by observing that by hypothesis the toric action commutes with the symplectomorphism defining the cosymplectic manifold.  Observe that the integrable system $F$ on the leaf extends to $Z$ only if $Z$ is a product or $F$ is invariant by the monodromy.
Denote by $(\alpha,\omega)$ the pair of 1 and 2-forms associated to the cosymplectic structure i.e, $\omega$ restricted to the symplectic leaves defines the symplectic structure on $Z$ and $\alpha$ is a closed form defining the codimension one symplectic foliation.

Following the extension theorem (Theorem 50 in \cite{GMP1}) we consider now the open $b$-symplectic manifold $U=Z\times (-\epsilon, \epsilon)$ with $b$-symplectic form,

$$\omega= \frac{df}{f}\wedge \pi^*(\alpha)+\pi^*(\omega)$$

where $\pi:U\to Z$ stands for the projection in the first component of $U$, $Z$ and $f$ is the defining function for the critical set $Z$.

Consider the map $\hat{F}=(c\log|t|, \pi^*(f_2), \dots, \pi^*(f_n))$  on with $c$ the modular period of $Z$ where we abuse notation and we write the components on the covering $L\times [0,1]$ of the mapping torus $Z$.

In this section we prove,

\begin{theorem}
The mapping $\hat{F}=(c\log|t|, \pi^*(f_2), \dots, \pi^*(f_n))$ defines a $b$-integrable system on the open $b$-symplectic manifold $Z\times (-\epsilon, \epsilon)$ thus extending the integrable system defined by the toric structure of $L$. The toric structure of $L$ extends to a toric structure on the $b$-symplectic manifold $Z\times (-\epsilon, \epsilon)$ if and only if the cosymplectic structure of $Z$ is trivial, i.e., $Z=L\times [0,1]$.
\end{theorem}

\begin{proof}

Observe that the functions $f_2, \dots, f_n$ define an integrable system on the cosymplectic manifold $Z$ as the gluing symplectomorphism that defines the mapping torus commutes with the torus action defined by $F$. So this torus action descends to the quotient $Z$ and the functions $f_i$ are well-defined on the mapping torus $Z$.
From the definition of $b$-symplectic form the projection $\pi$ is a Poisson map and thus
$\{ \pi^*(f_i), \pi^*(f_j)\}=\{ f_i, f_j\}=0$ for all $i,j\geq 2$. Observe also that functional independence on a dense set $W$ of $L$, of the functions  $f_2, \dots, f_n$ on $L$ (a factor of $U$) together with the functional independence of the pure $b$-function $c\log|t|$ from the functions $\pi^*(f_2), \dots, \pi^*(f_n)$ implies the functional independence on the dense open set $W\times I$ with the product topology.

Furthermore, the Poisson bracket $\{ c \log|t|, \pi^*(f_j)\}=0$ from the expression of $b$-symplectic structure. Thus the system $\hat{F}$ defines an integrable system on $Z\times (-\epsilon, \epsilon)$.

To conclude observe that the action-angle coordinates associated to the global toric action on $L$ extends to $Z$ (and thus to a neighborhood $Z\times (-\epsilon, \epsilon)$ if and only if the action extends to a toric action.
We now use Lemma \ref{lem:product} above to conclude that the toric structure extends to $Z$ if and only if the mapping torus is trivial, i.e., $Z=L\times [0,1]$.
This ends the proof of the theorem.
\end{proof}

Observe that given any cosymplectic compact manifold $Z$, then following the construction from Example 19 in \cite{GMP1}, $Z\times S^1$ admits a $b$-symplectic structure simply by considering the dual  $b$-Poisson structure (where $\pi$ is the corank regular Poisson structure associated to the cosymplectic structure and $X$ is a Poisson vector field transverse to the symplectic foliation in $Z$ as it was proved in \cite{GMP2}). The function $f$ is a function vanishing linearly. The critical locus of this $b$-Poisson structure has as many copies of the original $Z$ as zeros of the function $f$.

$$\Pi= f(\theta)X\wedge \frac{\partial}{\partial \theta}+\pi$$
The theorem above admits its compact version:

\begin{theorem}\label{thm:globalaa1}
The mapping $\hat{F}=(c\log|f(\theta)|, \pi^*(f_2), \dots, \pi^*(f_n))$ defines a $b$-integrable system on the  $b$-symplectic manifold $Z\times S^1$ thus extending the integrable system defined by the toric structure of $L$. The toric structure of $L$ extends to a toric structure on the $b$-symplectic manifold $Z\times S^1 $ if and only if the cosymplectic structure of $Z$ is trivial, i.e., $Z=L\times S^1$.
\end{theorem}

As a corollary we can detect situations in which the topological obstruction to global existence of action-angle coordinates lies in the non-triviality of the mapping torus defined by the critical set $Z$.

\begin{theorem}\label{thm:globalaa2}
Any $b$-integrable system  on $b$-symplectic manifold extending a toric system on a symplectic leaf of $Z$ does not admit global action-angle coordinates whenever the critical set $Z$ is not a trivial mapping torus $Z=L\times S^1$.
\end{theorem}

{ Below we show an example of a $b$-symplectic manifold $M$ of dimension $6$ an admits some $b$-integrable system which is not toric even though the leaves on the critical hypersurface are toric. Observe that the $b$-integrable system cannot define a toric action (and thus admit global action-angle coordinates) because of the topological structure of $Z$.}

\begin{ex}[Topological obstructions to semi-local action-angle coordinates]
{
Consider a product of spheres $S^2 \times S^2$ with coordinates $(h_1,\theta_2,h_2,\theta_2)$ and standard product symplectic form $\omega= dh_1 \wedge d\theta_1 + dh_2 \wedge d\theta_2$. The map
\begin{align*}
\varphi: S^2 \times S^2 &\longrightarrow S^2 \times S^2 \\
(p,q) &\longmapsto (q,p)
\end{align*}
is a symplectomorphism satisfying that $\varphi^2=\operatorname{Id}$. The induced map in homology swaps the generators of $H_2(S^2\times S^2)\cong H_2(S^2) \oplus H_2(S^2)$. This shows that $\varphi$ is not in the connected component of the identity, as this would imply that induced map in homology would act trivially \cite[Theorem 2.10]{Hat}. Thus, the mapping torus with gluing diffeomorphism $\varphi$ cannot be  a trivial product $S^2\times S^2 \times S^1$.

The pair of functions $F=(f_1,f_2)=(h_1+h_2,h_1h_2)$ are invariant with respect to $\varphi$ and hence descend to the mapping torus. Furthermore, they define { an integrable system (and in fact a toric action)} on $S^2\times S^2$, since they clearly Poisson commute and satisfy that $df_1\wedge df_2= (h_2-h_2)dh_1\wedge dh_2\neq 0$ almost everywhere. In particular,  by Remark \ref{rem:highdimen}, any $b$-symplectic manifold with critical set $Z$ admits a $b$-integrable system. However since the critical hypersurface is not a trivial product, any $b$-integrable system will not be toric in a neighborhood of $Z$.

By the discussion before the statement of Theorem \ref{thm:globalaa1}, the cosymplectic manifold $N$ can be realized as a connected component of a critical hypersurface of a compact $b$-symplectic manifold diffeomorphic to $M=N\times S^1$. Thus any $b$-integrable system in $M$ will not be toric even in a neighborhood of $Z$.

}
\end{ex}

Observe that with the magic  trick  of the desingularization  we obtain examples of folded-integrable systems without global action-angle coordinates. This is done  by applying Theorem \ref{thm:globalaa2} and the behaviour of torus actions under desingularization studied in \cite{convexitygmw}.

\begin{theorem}
  Let $F$ be a folded integrable system obtained by desingularization of a $b$-integrable system, if the critical set $Z$ of the original $b$-symplectic structure is not a trivial mapping torus, then the folded integrable system $F$ does not admit global action-angle coordinates.
\end{theorem}

Let us finish this article with a couple of concluding remarks:

\begin{itemize}

\item  For symplectic manifolds the obstructions to global action-angle coordinates started with Duistermaat in his seminal paper \cite{duistermaat} where Duistermaat related the existence of obstructions to the existence of monodromy which in its turn was naturally associated to the existence of singularities.

In this article we have concluded that for a singular symplectic manifold there are topological obstructions for existence of global action-angle coordinates that are detectable at first sight: The critical set $Z$ has to be  a trivial mapping torus $Z=L\times [0,1]$ thus the existence of monodromy associated to this mapping torus is also an obstruction.

\item Furthermore, the existence of action-angle coordinates yields a free action of a torus in the neighbourhood of a regular torus action thus the existence of isotropy groups for the candidate of torus action defining the system, automatically implies that the locus with non-trivial isotropy groups is singular for the integrable system. The same holds for a sub-circle. In particular:

\end{itemize}

\begin{corollary} Let $F$ be a $b$-integrable system as in Proposition \ref{s1action} on a $b$-symplectic manifold and denote by $T$ the union of the exceptional orbits of the $S^1$-action defined by $c\log |t|$.  Then the system has singularities at the set $T$.
\end{corollary}

Thus not only the topology of the critical set $Z$ yields an obstruction to existence of global action-angle coordinates but it also detects singularities of integrable systems. In particular along the exceptional orbits for the transverse $S^1$-action given by Proposition \ref{s1action}.
This motivates us to study singularities of integrable systems on singular symplectic manifolds, study which we will pursue in a different article.

\end{document}